\documentclass[12pt,twoside,a4paper,leqno]{article}
\usepackage{amsmath,amsfonts,amsthm,amssymb}
\usepackage{mathrsfs}
\usepackage{upref}
\usepackage{bm}

\usepackage{pgf}
\usepackage{graphicx,epsfig}
\usepackage{xcolor}

\usepackage{tikz}
\usepackage{tikz-cd}
\usetikzlibrary{matrix,arrows,decorations.pathmorphing}
\usetikzlibrary{positioning}

\usepackage[utf8]{inputenc}
\usepackage[T1]{fontenc}
\usepackage{lmodern}
\usepackage{verbatim}
\usepackage{enumitem}

\usepackage{makeidx}
\usepackage{index}
\newindex{sym}{symx}{symd}{Symboles}
\newindex{mc}{mcx}{mcd}{Mots-Clef}

\allowdisplaybreaks
\usepackage{fancyhdr}

\theoremstyle{plain}
\newtheorem{thm}{Theorem}[section]
\newtheorem{lm}[thm]{Lemma}
\newtheorem{prop}[thm]{Proposition}

\newtheorem{defn}[thm]{Definition}

\newtheorem{prop-def}[thm]{Proposition-definition}

\theoremstyle{plain}
\newtheorem*{thm*}{Theorem}
\newtheorem*{lm*}{Lemma}
\newtheorem*{prop*}{Proposition}
\newtheorem*{cor*}{Corollary}
\newtheorem*{defn*}{Definition}
\newtheorem*{conj*}{Conjecture}

\theoremstyle{definition}

\theoremstyle{remark}

\newtheorem{ex}{\bfseries Example}
\newtheorem{exs}[ex]{\bfseries Examples}

\theoremstyle{remark}
\newtheorem*{rem*}{\bfseries Remark}
\newtheorem*{rems*}{\bfseries Remarks}

\newcommand{\field}[1]{\mathbb{#1}}

\newcommand{\Qfield}{\field{Q}}

\newcommand{\Zfield}{\field{Z}}

\newcommand{\Ecal}{\mathcal{E}}
\newcommand{\Fcal}{\mathcal{F}}
\newcommand{\Gcal}{\mathcal{G}}

\DeclareMathOperator{\Br}{Br}

\DeclareMathOperator{\charac}{char}

\DeclareMathOperator{\coker}{coker}
\DeclareMathOperator{\Coker}{Coker}

\DeclareMathOperator{\Gal}{Gal}

\DeclareMathOperator{\Hom}{Hom}
\DeclareMathOperator{\Id}{Id}
\DeclareMathOperator{\Ind}{Ind}
\DeclareMathOperator{\Ima}{Im}
\DeclareMathOperator{\Ker}{Ker}

\DeclareMathOperator{\Nor}{N}

\DeclareMathOperator{\divi}{\mid}

\DeclareMathOperator{\ndivi}{\nmid}

\title{\Large On Greenberg's generalized conjecture for families of number fields}
\author{ by \large NGUYEN QUANG DO Thong}
\date{October 11, 2025}

\begin{document} 
%\pagestyle{fancy}
%\fancyhf{} \fancyhead[C]{\thepage}
%\setlength{\headheight}{15pt}
%\linenumbers
\begin{center}
{\large On Greenberg's generalized conjecture \\
for families of number fields}\\
{by Thong NGUYEN QUANG DO}\\
October 11, 2025
\end{center}

%%
%%abstract
%%
\begin{abstract} For a number field $k$ and an odd prime $p$, let $\tilde{k}$ be the compositum of all the $\Zfield_p$-extensions of $k$, $\tilde{\Lambda }$ the associated Iwasawa algebra, and $X(\tilde{k})$ the Galois group over $\tilde{k}$ of the maximal abelian unramified pro-$p$-extension of $\tilde{k}$. Greenberg's generalized conjecture (\emph{GGC} for short) asserts that the $\tilde{\Lambda}$-module $X(\tilde{k})$ is pseudo-null. Very few theoritical results toward \emph{GGC} are known. We show here that for an imaginary k, \emph{GGC} is implied by certain pseudo-nullity conditions imposed on a special $\Zfield^2_p$-extension of $k$, and these conditions are partially or entirely fullfilled by certain families of number fields.
\end{abstract}
%%
%%section
%%
\section{\large\label{sec1} Introduction.} For  a number field $k$ and an odd prime $p$, for any integer $d\geq 1$, denote by $K^{(d)}$ a  $\Zfield^d_p$-extension of $k$, $\Gamma^{(d)}=\Gal\big(K^{(d)}/k\big)$, $\Lambda^{(d)}=\Zfield_p[[\Gamma^{(d)}]]$, $X(K^{(d)})$ the Galois group over $K^{(d)}$ of the maximal abelian unramified pro-$p$-extension of $K^{(d)}$. Two important particular cases are $k^{\text{cyc}}$, the cyclotomic  $\Zfield_p$-extension of $k$, and $\tilde{k}$, the compositum of all the $\Zfield_p$-extensions of $k$. If $k$ is totally real, Greenberg's classical conjecture \emph{GC} for short), which is a ``reasonable'' generalization of Vandiver's conjecture  for $\Qfield_p(\mu_p)^+$, asserts that the $\Lambda^{\text{cyc}}$-module $X(k^{\text{(cyc)}})$ is finite (see \cite{rgre1}).
%%
%%paragraph
%%
\paragraph{Geenberg's generalized conjecture (\emph{GGC} for short)}{\ }\\
For any $k$ , \emph{GGC} states that \emph{the $\tilde{\Lambda}$-module $X(\tilde{k})$ is pseudo-null, i.e. its $\tilde{\Lambda}$-annihilator has height $\geq 2$} (\cite{rgre2}).  Note that for a totally real field which satisfies Leopoldt's conjecture, \emph{GC} and \emph{GGC} coincide.\medskip

Although \emph{GC} has been checked for many families of numerical examples, mainly families of real quadratic fields with $p=3$, and although theoritical criteria exist (\cite{rgre1}, \cite{tngu2}), practically no general positive result is known, but see \cite{tngu1} and \cite{tngu2}. Curiously, the situation concerning \emph{GGC} for imaginary number fields is a little better. Let us cite R. Sharifi's result for $\Qfield(\mu_p)$ in the current range $p< 25.10^3$ (\cite{mo_sh1}), § 10, theorem~2; see also example 2 below ), and S. Fujii's for imaginary CM-fields verifying Leopoldt's conjecture  and the so called Itoh conditions (\cite{sfuj1}, thm 2; see also the appendix in §6 below). Sharifi's approach relies on remarkable relationship between modular symbols and values of cup-products of cyclotomic $p$-units, and Fujii's on elaborate constructions from class field  theory.\medskip

Fujii's theorem  encompasses previous results on quadratic and quartic abelian fields (\cite{jmin1}, \cite{tito1}) but the drawback of the stringent Itoh conditions is that we do not know whether there are infinitely many examples of fields satisfying them. Let us replace $X(K^{(d)})$ by the more natural (in an \'etale setting) module $X'(K^{(d)})$ defined as being the Galois group aver $K^{(d)}$ of the maximal abelian unramified pro-$p$-extension of $K^{(d)}$ which splits totally all the $p$-primes of $K^{(d)}$; let \emph{GGC'} be the conjecture obtained from \emph{GGC} on replacing $X(K^{(d)})$ by  $X'(K^{(d)})$, and which is equivalent to it if a certain decomposition property  (\emph{Dec}) holds for  $K^{(d)}/k$ (see details in section~\ref{sec5}). Besides, while \cite{sfuj1}  relies on Leopoldt's conjecture, we shall also appeal systematically to the Kuz'min-Gross conjecture in a certain way, which we shall loosely refer to as ``\emph{modulo the Kuz'min-Gross conjecture}''  (see \S\S \ref{sec5} and \ref{sec10} for details).
%%
%%paragraph
%%
\paragraph{The Kuz'min-Gross conjecture}{(\cite{fe_gr1}, \cite{ikuz1})}{\ }\\
For a number field $K$ and its cyclotomic $\Zfield_p$-extension $K^{\text{cyc}}$, with Galois group $\Gamma=\Gal(K^{\text{cyc}}/K)$, define as previously $X'(K^{\text{cyc}})$ to be the Galois group over $K^{\text{cyc}}$ of the maximal abelian unramified pro-$p$-extension of $K^{\text{cyc}}$ which split totally all the $p$-primes of $K^{\text{cyc}}$, and $X'(K^{\text{cyc}})^0$ the maximal finite submodule of $X'(K^{\text{cyc}})$. The Kuz'min-Gross conjecture for $K$ predicts the \emph{finiteness} of the module of invariants $X'(K^{\text{cyc}})^{\Gamma}$ (or equivalently of the module of coinvariants $X'(K^{\text{cyc}})_{\Gamma}$). If it holds $X'(K^{\text{cyc}})^{\Gamma}=\big(X'(K^{\text{cyc}})^0\big)^{\Gamma}$ has the same order as $\big(X'(K^{\text{cyc}})^0\big)_{\Gamma}$, but $\big(X'(K^{\text{cyc}})^0\big)_{\Gamma}\neq X'(K^{\text{cyc}})_{\Gamma}$ in general.\medskip

Away from Fujii's setting (see appendix), we aim to give general criteria by showing three main results modulo (\emph{Dec)} and the Kuz'min-Gross conjecture:
\begin{description}
\item[A.] If an imaginary number field $k$ admits a special $\Zfield_p^2$-extension $K^{(2)}$ such that the $\Lambda^{(2)}$-module $X'\big(K^{(2)}\big)$ is pseudo-null, then $k$ verify \emph{GGC} (theorem~\ref{t:5-5}).
\item[B.] For a special $\Zfield_p^2$-extension $K^{(2)}/k$ obtained by composing $k^{\text{cyc}}$ with an independant  $\Zfield_p$-extension $F_{\infty}/k$, with $\Gamma=\Gal(k^{\text{cyc}}/k)\simeq \Gal\big(K^{(2)}/F_{\infty}\big)$, $\Lambda=\Zfield_p[[F_{\infty}/k]]$, we give a theorem computing the deviation between the $\Lambda$-characteristic series of $X'\big(K^{(2)}\big)_{\Gamma}$ and of $X'\big(K^{(2)}\big)^{\Gamma}$ (theorem~\ref{t:10-5}). This gives a criterion for the validity of \emph{GGC} (see theorem~\ref{t:15-10}).
\item[C.] We construct families of imaginary fields which satisfiy at least partially the above criterion (see \S\S\ref{sec20}-\ref{sec25}).
\end{description}

One particular family consists of the so called  $\bm{ (p,i)-{\rm regular}}$ fields (Itoh's conditions\label{itoh}, \cite{tito1}, imply that the base field $k$ is $(p,0)-{\rm regular}$, see \S\ref{sec30}), for which the ``$e_{i-1}$-part of \emph{GGC}'' can be proved (see theorem~\ref{t:25-5}). If $k$ is $(p,i)-{\rm regular}$, so is any extension $F/k$ which is unramified outside $p$. the notion of $ (p,i)-{\rm regularity}$ being designed to ``kill a part of the arithmetic of the field at the prime $p$'',  it would be probably possible to show algebraically that  $(p,i)-{\rm regular}$ fields verify  the Kuz'min-Gross conjecture. Anyway, since this conjecture (at the same title as Leopoldt's) is proved for abelian fields and widely expected to hold true in general, its intervention in the hypothesis is ``morally superfluous''.
%%
%%section
%%
\section{\large\label{sec5} Induction on $d\geq 2$.}
From now on, $k$ denotes an imaginary field with $r_2$ complex places. It is classically known that $\widetilde{\Gamma}=\Gal(\widetilde{k}/k)$ is a free $\Zfield_p$-module of rank $1+r_2+\delta$, where $\delta\geq 0$ is the defect of Leopoldt's conjecture. recall our two types of Iwasawa modules:\\
$X(K^{(d)})$ (resp. $X'(K^{(d)})$)= the Galois group over $K^{(d)}$ of the maximal abelian pro-$p$-extension of $K^{(d)}$ which is unramified everywhere (resp. which splits completely every $p$-place of $K^{(d)}$), and the two Greenberg's generalized conjectures, with $\widetilde{\Lambda}=\Zfield_p[[\widetilde{\Gamma}]]$:\\
\emph{GGC} The $\tilde{\Lambda}$-module $X(\tilde{k})$ is pseudo-null\\
 \emph{GGC'} The $\tilde{\Lambda}$-module $X'(\tilde{k})$ is pseudo-null.\medskip

\emph{GGC} implies \emph{GGC'} because $X'(\tilde{k})$ is a quotient of $X(\tilde{k})$. For any number  field $K$, denote by $U_K$ (resp $U_{K}'$) the group of units (resp. $(p)$-units) of $K$, by $A_K$ (resp. $A'_K$) the $p$-group of ideal classes (resp. of $(p)$-ideal classes) of $K$ , and by $S$ the set of $p$-primes of $K$. The classical exact sequence induced by valuations 
\begin{equation*}
0\rightarrow U_K\rightarrow U'_K\rightarrow \Zfield_p[S]\rightarrow A_K\rightarrow A'_K\rightarrow 0
\end{equation*}
is used to show that \emph{GGC} implies \emph{GGC'} if $\tilde{k}$ satisfies the following decomposition condition:\medskip

\begin{center}
\begin{minipage}{80mm}
\emph{(D)  For any $p$-place $v$ of $\tilde{k}$, the decomposition subgroup $\tilde{\Gamma}_v$ has $\Zfield_p$-rank at least 2.}
\end{minipage}
\end{center}

It seems likely that (D) is valid for imaginary base field $k$-in general or under mild conditions- but any attempt to prove this gets bogged down by non trivial problems of $p$-adic independence when the extension $k/\Qfield$ is not normal. In the sequel we shall impose the following analogous decomposition property on certain $\Zfield_p^d$-subextensions of $\tilde{k}$:
\begin{center}
\begin{minipage}{80mm}
\emph{(Dec) $K^{(d)}/k$ is a $\Zfield^d_p$-extension, with $d\geq 2$, such that  for every $p$-place $v$ of $K^{(d)}$,  the corresponding decomposition subgroup $\Gamma^{(d)}_v$ has $\Zfield_p$-rank at least 2.}
\end{minipage}
\end{center}
One can construct examples from the following criterion:
%%
%%Lemma
%%
\begin{lm}\label{l:5-1}
Let $k$ be imaginary, and let $K^{(d)}/k$ be a $\Zfield^d_p$-extension ($d\geq 2$) which is normal over $\Qfield$. Then for any $p$-place $v$ of $K^{(d)}$,  $\Zfield_p$-rank $\Gamma^{(d)}_v\geq d/s$, where $s$ is the number of $p$-places  of $k$.
\end{lm}
%%
%%proof
%%
\begin{proof} We adopt the point of view of \cite{fl_ng1}: denote by $S$ the set of $p$-primes of $k$, by $k_S$ the maximal algebraic extension of $k$ which is unramified outside $S$, $G_S(k)=\Gal(k_S/k)$; the $\Zfield_p$-extensions of $k$ are classified by $H^1(G_S(k),\Zfield_p)=\Hom(G_S(k),\Zfield_p)=\Hom(G_S(k)^{\text{ab}},\Zfield_p)$, and this group in turn can be described by the Poitou-Tate exact sequence (see \cite{psch1}, \S2, Satz 4):
\begin{multline*}
\bigoplus_{v\in S}H^1\big(k_v,\Qfield_p/\Zfield_p(1)\big)\rightarrow H^1(G_S(k),\Zfield_p)^{\vee}\rightarrow H^2\big(G_S(k),\Qfield/\Zfield_p(1)\big)\rightarrow\\
\rightarrow\bigoplus_{v\in S} H^2\big(k_v,\Qfield_p/\Zfield_p(1)\big)
\end{multline*}
where $(.)^{\vee}$ denotes the Pontryagin dual. The kernel of the map
\begin{equation*}
H^2\big(G_S(k),\Qfield_p/\Zfield_p(1)\big)\rightarrow\bigoplus_{v\in S} H^2\big(k_v,\Qfield_p/\Zfield_p(1)\big)
\end{equation*}
is dealt with the theory of Brauer groups; according to \cite{psch1}, \S 4, lemma 2, this is nothing but 
\begin{equation*}
\Ker\big(\Br(k)\{p\}\rightarrow\bigoplus_{\text{all }v}\Br(k_v)\{p\}\big)
\end{equation*}
which is null by the Brauer-Hasse-Noether theorem. Taking Pontryagin duals, we then get a short exact sequence
\begin{equation*}
0\rightarrow H^1(G_S(k),\Zfield_p)\rightarrow\bigoplus_{v\in S}H^1\big(k_v,\Qfield_p/\Zfield_p(1)\big)^{\vee} \cong \bigoplus_{v\in S}H^1\big(k_v,\Zfield_p\big)
\end{equation*}
(by local duality),which means that the natural cohomological map
\begin{align*}
j:\ H^1(G_S(k),\Zfield_p) & \rightarrow \bigoplus_{v\in S}H^1\big(k_v,\Zfield_p\big)\\
 \varphi & \mapsto\big(j_v(\varphi)\big)_{v\in S}
\end{align*}
is injective. Recall that the map $j_v$ consists in restricting first from $G_S(k)$ to a decomposition subgroup $G_S(k)_v$, then inflating to the absolute Galois group of $k_v$. Let us assume from now on  that some given $\Zfield_p$-extension $K^{(d)}/k$ is normal over $\Qfield$, with $\Gamma^{(d)}= \Gal(K^{(d)}/k))$ ($d\geq 2$). The analog of $j$ is the natural map 
\begin{equation*}
j^{(d)}=\big(j^{(d)}_v\big)_{v\in S}: H^1\big(\Gamma^{(d)},\Zfield_p\big)\xrightarrow{\text{res}}\bigoplus_{v\in S}H^1\Gamma^{(d)}_v,\Zfield_p\big)\xrightarrow{\text{inf}}\bigoplus_{v\in S}H^1(k_v,\Zfield_p)\, .
\end{equation*}
The injectivity of $j$ and of the inflation maps 
\begin{equation*}
H^1\big(\Gamma^{(d)},\Zfield_p\big)\rightarrow H^1(G_S(k),\Zfield_p) \text{ and } H^1\big(\Gamma^{(d)}_v,\Zfield_p\big)\rightarrow H^1(G_S(k)_v,\Zfield_p)
\end{equation*}
obviously implies the injectivity of the above map $\text{\it res}$. Since $K^{(d)}/\Qfield$ is Galois, all the local maps $j_v^{(d)}$ behave in the same manner, and therefore  $\Zfield_p$-rank $H^1\big(\Gamma_v^{(d)},\Zfield_p\big)\geq d/s$ for any $p$-place $v$ of $K^{(d)}$.
\end{proof}
To go further in the direction of \emph{(Dec)}, let us now restrict to extensions \hfill\break
$K^{(d)}/k^{\text{cyc}}/k$.
%%
%%lemma
%%
\begin{lm}[cp. \cite{tkat1}, lemma 4.10]\label{l:5-5}
Let $k$ be an imaginary number field, normal over $\Qfield$, admitting a $\Zfield_p^2$-extension $K^{(2)}$ which contains $k^{\text{cyc}}$ and is normal over $\Qfield$ (such a $\Zfield_p^2$-extension will be called \emph{``special''}). Suppose also that $k$ satisfies the Kuz'min-Gross conjecture. Then $K^{(2)}$ verifies  (Dec), and so does any $\Zfield_p^d$-extension containing $K^{(2)}$.
\end{lm}
%%
%%proof
%%
\begin{proof} Let $L'$ be the maximal abelian pro-p-extension of $k$ containing $k^{\text{cyc}}$ and such that all the $p$-places of $k^{\text{cyc}}$ are totally split in $L'$. Denote $\Gamma=\Gal(k^{\text{cyc}}/k)$; then, by class field theory, $\Gal(L'/k^{\text{cyc}})$ is no other than $X'(k^{\text{cyc}})_{\Gamma}$, which is finite according to Kuz'min-Gross. This implies that there exists at least one $p$-place $w$ of $k^{\text{cyc}}$ which is only finitely decomposed in $K^{(2)}/k^{\text{cyc}}$, i.e. for which $\Gal(K^{(2)}/k^{\text{cyc}})_w$ has $\Zfield_p$-rank $1$, and this happens for all $w$ above $p$ since $K^{(2)}$ is normal aver $\Qfield$. Hence $\Gamma^{(2)}_v$ has $\Zfield_p$-rank $2$ for all $p$-places $v$ of $k$. The rest of the lemma is obvious.
\end{proof}
%%
%%exemple
%%
\begin{ex}\label{ex:5-1} Obviously any imaginary quadratic field admits a special $\Zfield_p^2$-extension, the compositum of its cyclotomic and anti-cyclotomic $\Zfield_p$-exten\-sions. Then any $k$ containing an imaginary quadratic field will admit a  special $\Zfield_p^2$-extension by translation. This can be extended as follows.
Suppose that $k$ is an abelian CM-field over $\Qfield$ and the exponent of $G=\Gal(k/\Qfield)$ divides $p-1$. Using the theory of pro-$p$-groups with $G$-action, Greenberg has shown that for every odd character $\chi$ of $G$ (viewed as taking its values in $\Zfield^*_p$), there exists a uniquely determined $\Zfield_p$-extension such that $G$ acts on $\Gamma_{\chi}=\Gal(k^{\chi}_{\infty}/k)$ by the character $\chi$; the extension $k^{\chi}_{\infty}/k$ is normal, with Galois group isomorphic to $\Gamma_{\chi} \rtimes G$ (\cite{rgre3}, remark 3.3.2). It follows that the compositum of $k^{\text{cyc}}$ and $k^{\chi}_{\infty}$ is indeed a special $\Zfield_p^2$-extension of $k$. For convenience, we shall refer to it as a special $\Zfield_p^2$-extension \emph{``of Greenberg type''}.
\end{ex}

From now on we focus on \emph{GGC'}. The point is that the fate of \emph{GGC'} can be decided very early, e.g. starting from $\Zfield_p^2$-extensions of $k$ as in \cite{sfuj1}. Under the so called Itoh conditions for a CM field k (see the appendix) and using class field theory, S. Fujii constructed a sequence of (uniquely determined) multiple $\Zfield_p$-extensions $k\subset\dots \subset K^{(d)}\subset  K^{(d+1)}\subset\dots \subset \tilde{k}$, and he showed by elaborate class field theoritic calculations that if  $X(K^{(d)})$ is $\Lambda^{(d)}$-pseudo-null for $d\geq 2$, then $X(K^{(d+1)})$ is $\Lambda^{(d+1)}$-pseudo-null for $d\geq 2$ 'see \cite{sfuj1}, theorem~3 of step 3). In the sequel , we shall show for \emph{GGC'} results analogous to Fujii's, but in another direction. For induction steps, but also for the arguments in the next secxtion, we shall need the following algebraic criterion:
%%
%%lemma
%%
\begin{lm}[\cite{ba_ba_lo1}, 2.10 and 2.11] \label{l:5-10}Let $R$ be a noetherian Krull domain (e.g. $R=\Lambda^{(d)}$) and $S=R[[t]]$. Let $M$ be a noetherian torsion $S$-module such that $M/tM$ is a noetherian torsion $R$-module. Then the $t$-torsion  submodule $M_t$ is also a noetherian $R$-torsion module, $\charac_R(M_t)$ divides $\charac_R(M/tM)$, and $M$ is $S$-pseudo-null if and only if  $\charac_R(M_t)=\charac_R(M/tM)$, where $\charac_R(.)$ denotes the characteristic ideal. Moreover,, if $M/tM$ is $R$-pseudo-null, then $M$is $S$-pseudo-null.
\end{lm}
(note that the last sufficient condition in the lemma can already be found in \cite{jmin1}.)\medskip

To simplify the language, we shall say that a $\Zfield_p^d$-extension $K^{(d)}/k$ is \emph{green} if the $\Lambda^{(d)}$-module $X'(K^{(d)})$ is pseudo-null. So \emph{GGC'} simply means that $\tilde{k}/k$ is green. Here is a refinement of the proposition~4.B of Minardi's thesis \cite{jmin1}:
%%
%%theorem
%%
\begin{thm}\label{t:5-5} If an imaginary number field $k$, normal over $\Qfield$, admits a $\Zfield_p^2$-extension which is both special (see the definition in lemma ~\ref{l:5-1}) and green, then $k$ verifies \emph{GGC'}.
\end{thm}
%%
%%begin proof
%%
\begin{proof} Starting from given $\Zfield_p^2$-extension, let us build a particular tower of multiple $\Zfield_p$-extensions
\begin{equation*}
K^{(2)}\subset \dots \subset K^{(d)}\subset K^{(d+1)}\subset \dots\subset \tilde{k}\, .
\end{equation*}
To ease the notations, write ${X'}^{(d)}$ for $X'(K^{(d)})$. Since $d+1\geq 2$, it is known (\cite{ba_ba_lo1}, lemma 3.2) that one can choose the tower in such a way that the co-invariant module $({X'}^{(d+1)})_{\Gamma}$, where $\Gamma=\Gal\big(K^{(d+1)}/K^{(d)}\big)\cong \Zfield_p$, is $\Lambda^{(d)}$-torsion. We want to show that if $K^{(d)}$ is green, then  $K^{(d+1)}$ is too. By the sufficient condition in lemma~\ref{l:5-10}, it is enough to show that $({X'}^{(d+1)})_{\Gamma}$ is $\Lambda^{(d)}$-pseudo-null. Let $Y^{(d+1)}=Y(K^{(d+1)})$ denote the Galois group over $K^{(d+1)}$ of the maximal abelian pro-p-extension of $K^{(d+1)}$ which is unramified outside $(p)$. By definition, we have an exact sequence of $\Lambda^{(d+1)}$-modules
\begin{equation*}
0\rightarrow D^{(d+1)}\rightarrow Y^{(d+1)}\rightarrow ({X'}^{(d+1)})\rightarrow 0
\end{equation*}
where $D^{(d+1)}=D\big(K^{(d+1)}\big)$ denotes the decomposition subgroup relative to all $p$-places of $K^{(d+1)}$. The snake lemma yields an exact sequence of invariant/co-invariant modules
\begin{equation*}
\dots \rightarrow\big({X'}^{(d+1)}\big)^{\Gamma}\xrightarrow{\varphi}\big(D^{(d+1)}\big)_{\Gamma}\rightarrow\big(Y^{(d+1)}\big)_{\Gamma}\rightarrow \big({X'}^{(d+1)}\big)_{\Gamma}\rightarrow 0
\end{equation*}
which can be inserted into a commutative diagram of $\Lambda^{(d)}$-modules:
\begin{center}
\begin{tikzcd}
0\arrow[r] & \coker \varphi\arrow[d,"\psi_1"] \arrow[r] & \big(Y^{(d+1)}\big)_{\Gamma}  \arrow[d,"\psi_2"] \arrow[r]  & \big({X'}^{(d+1)}\big)_{\Gamma} \arrow[d,"\psi_3"] \arrow[r] & 0\\
0\arrow[r] & D^{(d)}\arrow[r] &  Y^{(d)}\arrow[d]\arrow[r] & {X'}^{(d)}\arrow[r] & 0\\
& & \Zfield_p &&
\end{tikzcd}
\end{center}
the right vertical arrow $\psi_3$ is the natural one, and the left vertical arrow $\psi_1$ is induced by the natural map $\psi_4: D^{(d+1)}\rightarrow  D^{(d)}$. Let us explain the middle arrow $\psi_2$. Making a slight abuse of language, denote by $S$ the set of $p$-primes of any extension of $k$, and by $k_S$ the maximal algebraic extension of $k$ unramified outside $S$. Note that $k_S$ contains $ K^{(d)}$.

Writing $G_S\big( K^{(d)}\big)=\Gal\big(k_S/ K^{(d)}\big)$, our middle vertical column is just the Pontryagin dual of the inflation-restriction sequence with coefficients $\Qfield_p/\Zfield_p$
\begin{multline*}
0\rightarrow H^1(\Gamma,\Qfield_p/\Zfield_p)\rightarrow H^1\big(G_S\big( K^{(d)}\big),\Qfield_p/\Zfield_p\big)\rightarrow\\
\rightarrow H^1\big(G_S\big( K^{(d+1)}\big),\Qfield_p/\Zfield_p\big)^{\Gamma}\rightarrow H^2(\Gamma,\Qfield_p/\Zfield_p)=0\, .
\end{multline*}
In particular $\psi_2$ is injective and has cokernel $\Zfield_p$. Since $\Zfield_p$ is pseudo-null because $d\geq 2$, a simple diagram chase shows that $\Ker \psi_3$ and $\Coker \psi_1$ are pseudo-isomorphic and $\Coker\psi_3$ is pseudo-null over $\Lambda^{(d)}$. Since ${X'}^{('d)}$ is pseudo-null by hypothesis, the pseudo-nullity of $\big({X'}^{(d+1)}\big)_{\Gamma}$ is then equivalent to that of $\Coker \psi_1$. Let us recall the description of $D^{(d+1)}$ by class field theory. For any finite subextension $F$ of $K^{(d+1)}$, let $U'_F$ be the group of ($p$)-units of $F$, $\overline{U_F'}=U'_F\otimes\Zfield_p$, $\overline{F_v^*}$ the $p$-adic completion of $F^*_v$, $\Fcal=\bigoplus_{v\divi p} \overline{F_v^*}$, $X'_F$, $Y_F$ and $D_F$ the analogues at the level $F$ of ${X'}^{(d+1)}$, $Y^{(d+1)}$ and $D^{(d+1)}$. The exact sequence of class field theory relative to decomposition at the level $F$ reads:
\begin{center}
\begin{tikzcd}
\overline{U_F'}\arrow[r] & \Fcal \arrow[rr] \arrow[dr]& &Y_F   \arrow[r]  & X'_F  \arrow[r] & 0\\
& & D_F  \arrow[ur]&&
\end{tikzcd}
\end{center}
Taking the inverse limit with respect to the finite subextensions $F$ of $K^{(d+1)}$, we get a sequence
\begin{equation*}
\overline{U'}^{(d+1)}\rightarrow \Fcal^{(d+1)}\rightarrow D^{(d+1)}\rightarrow 0
\end{equation*}
where 
\begin{equation*}
\overline{U'}^{(d+1)}=\varprojlim\overline{U_F'}\text{ and }\Fcal^{(d+1)}=\bigoplus_{v\divi p}\Ind_v\Fcal_v^{(d+1)}
\end{equation*}
where $\Ind_v$ denotes the induced module from $\Gamma_v^{(d+1)}$ to $\Gamma^{(d+1)}$, and $\Fcal_v^{(d+1)}=\varprojlim\overline{F_v^*}$. As previously, the snake lemma yields a commutative diagram
\begin{center}
\begin{tikzcd}
\arrow[r] & \overline{U'}^{(d+1)}\arrow[d] \arrow[r] & \big(\Fcal^{(d+1)}\big)_{\Gamma}  \arrow[d,"\psi_4"] \arrow[r]  & \big(D^{(d+1)}\big)_{\Gamma} \arrow[d,"\psi_5"] \arrow[r] & 0\\
\arrow[r] & \overline{U'}^{(d)}\arrow[r] &  \Fcal^{(d)} \arrow[r] & D^{(d)}\arrow[r] & 0
\end{tikzcd}
\end{center}
The module $(\Fcal^{(d+1)}_v$ being the local analogue of $Y^{(d+1)}$, the same inflation-restriction argument as previously yields an exact sequence
\begin{equation*}
0\rightarrow \big(\Fcal_v^{(d+1)}\big)_{\Gamma_v}\rightarrow \Fcal_v^{(d)}\rightarrow\Gamma_v\rightarrow 0
\end{equation*}
where $\Gamma_v$ is null or isomorphic to $\Zfield_p$. But $\Coker \psi_4$ is a quotient  of $\bigoplus_{v\divi p}\Gamma_v$ and each $\Gamma_v$ is killed by the augmentation ideal of $\Zfield_p[[\Gamma^{(d+1)}_v]]$ (simply because $\Gamma^{(d+1)}_v$ is abelian), so the pseudo-nullity of $\Coker\psi_4$ is a consequence of (\emph{Dec})  for $K^{(d+1)}$ (see lemma~\ref{l:5-5}). It follows that $\Coker \psi_1$, as a quotient of $\Coker\psi_5$, which is itself a quotient of $\Coker\psi_4$, is actually pseudo-null and we are done.
\end{proof}
%%
%%section
%%
\section{\large\label{sec10} Green $\Zfield_p^2$-extensions.} Recall that \emph{green} was defined just before in theorem~\ref{t:5-5}. In order to apply this theorem we must construct a green $\Zfield_p^2$-extension of $k$. The previous induction argument works no more because for $d=1$, the module $\Zfield_p$ is no longer pseudo-null. Anyway it cannot work in general since not all $\Zfield_p^2$-extensions are green. Under Itoh's conditions, (page~\pageref{itoh}), Fujii showed, again by elaborate class field theoritic calculations in his (uniquely defined) tower of extensions $k\subset K^{(1)}\subset\dots\subset K^{(d)}\subset \dots \subset \tilde{k}$, that $X^{(1)}$ is trivial and $X^{(2)}$ pseudo-null (\cite{sfuj1}, step 1 and proposition 1 of step 2). We aim to produce another criterion in a different setting, but before proceeding, let us fix some notations and conventions.
%%
%%definition
%%
\begin{defn}\label{d:10-5} The Galois setting is the following: $K^{(2)}/k$ is the compositum of the cyclotomic  $\Zfield_p$-extension $k^{\text{cyc}}=\cup k_n$ and of an independant $\Zfield_p$-extension $F_{\infty}=\cup F_m$; by definition $k^{\text{cyc}}\cap F_{\infty}=k$, but this is not a restriction since for any subgroup $G$ of finite index in $\Gamma^{(2)}$, pseudo-nullity over $G$ implies pseudo-nullity over $\Gamma^{(2)}$ (see e.g. \cite{sfuj1}, lemma 6).
\end{defn}
%%
%%paragraph
%%
\paragraph{Extension maps.} For all $m$, denote $\Gamma= \Gal\big(k^{\text{cyc)}}/k\big)\cong \Gal\big(F^{\text{cyc}}_m/F_m\big)$; for all $m$, $n$, letting $K_{m,n}=F_m\cdot k_n$, write 
\begin{equation*}
\Gal\big(F^{\text{cyc}}_m/K_{m,n}\big)=\Gamma_{m,n}\cong \Gamma^{p^n}=  \Gal\big(k^{\text{cyc)}}/k\big)=\varinjlim_n A'(K_{m,n})\, .
\end{equation*}
and
\begin{equation*}
G_{m,n}=\Gal(K_{m,n}/F_m)\cong G_n=\Gal(k_n,k)
\end{equation*}
(it is recommended to draw a Galois diagram). Denote by $A'(K_{m,n})$ the $p$-class group of the ring of $(p)$-integers of $K_{m,n}$, $A'(F_m^{\text{cyc}})=\varinjlim_n A'(K_{mn,})$. Recall that between two levels $h\geq n$, the extension map 
\begin{equation*}
j_{n,h}:\  A'(k_n)\rightarrow A'(k_h)
\end{equation*}
induced by extension of ideals is compatible with the natural map 
\begin{equation*}
X'\big(k^{\text{cyc}}\big)_{\Gamma^{p^n}}\rightarrow X'\big(k^{\text{cyc}}\big)_{\Gamma^{p^h}}
\end{equation*}
 induced by $\displaystyle v_{h,n}=\frac{\gamma^{p^h}-1}{\gamma^{p^n}-1}$ in $X'\big(k^{\text{cyc}}\big)$, where $\gamma$ is a topological generator of $\Gamma$.
%%
%%paragraph
%%
\paragraph{Norm maps.} At finite level $h\geq n$, the arithmetic norm $\Nor_{h,n}:\ A'(k_h)\rightarrow A'(k,n)$ is related to the extension map $j_{n,h}$ and to the algebraic norm $v_h:\ A'(k_h)\rightarrow A'(k_h)$ by $\Nor_{h,n}\circ j_{n,h}=v_h$. We need to introduce norm maps between some ``obstruction kernels''. For any $n\geq 0$, recall that class field theory gives a natural map $X'\big(k^{\text{cyc}}\big)_{\Gamma^{p^n}}\rightarrow A'n$ which is surjective for $n\gg 0$; let $\Psi_n=\Psi(k_n)$ be its kernel and $\Psi\big(k^{\text{cyc}}\big)=\varinjlim_n\Psi_n$ w.r.t. the extension maps. It will be recalled in proposition~\ref{p:10-5} below that the $\Psi_n$'s stabilize by extension and become isomorphic to $\Psi\big(k^{\text{cyc}}\big)$ for $n\gg 0$.\medskip

We can now state our main technical result. Recall that $\lambda'$ and $\mu'$ (resp. $\lambda$ and $\mu$) are the Iwasawa invariants relative to $X'(.)$  (resp. $X(.)$).
%%
%%theorem
%%
\begin{thm}\label{t:10-5} Let $K^{(2)}$ be a special, (lemma~\ref{l:5-5}), $\Zfield_p^2$-extension of the normal imaginary number field $k$, obtained by composing  $k^{\text{cyc}}$ with an independant $\Zfield_p$-extension $F_{\infty}=\cup F_m$ such that $k$ and all the fields $F_m$ verify the Kuz'min-Gross conjecture (``modulo the Kuz'min-Gross conjecture'' for short). \medskip

If $ \varprojlim_m\big(X'\big(F_m^{\text{cyc}}\big)^0\big)^{\Gamma}$ and $\varprojlim_m\Psi\big(F^{\text{cyc}}_m\big)$ are finite, then $K^{(2)}/k$ is green (in the terminology of theorem ~\ref{t:5-5}) if and only if  $\varprojlim_m A'(F_m)=X'(F_{\infty})$ is finite i.e. the Iwasawa  invariants $\lambda'(F_{\infty})$ and $\mu'F_{\infty})$ are simultaneously null.
\end{thm}
(recall that \emph{special} just means that $K^{(2)}$ contains $k^{\text{cyc}}$ and is normal over $\Qfield$, see lemma~\ref{l:5-5}).\medskip

The proof will proceed in several steps. For simplicity, denote  by $\big({X'}^{(2)}\big)_{\Gamma}$ and $\big({X'}^{(2)}\big)^{\Gamma}$ respectively the modules  $X'\big(K^{(2)}\big)_{\Gamma}=\varprojlim_{m,n}A'\big(K_{m,n}\big)_{G_{m,n}}$ and $X'\big(K^{(2)}\big)^{\Gamma}=\varprojlim_{m,n}A'\big(K_{m,n}\big)^{G_{m,n}}$ (w.r.t. the norms). Since the property (\emph{Dec)}) is valid because $K^{(2)}$ is special, no $p$-place is totally decomposed in $F_{\infty}/k$, hence (see \cite{ba_ba_lo1}, lemma~3.2) $\big({X'}^{(2)}\big)_{\Gamma}$ is $\Zfield_p[[\Gamma]]$-torsion, and lemma~\ref{l:5-10} will be applicable.
%%
%%paragraph
%%
\paragraph{Miscellaneous results on (co)capitulation} { \ }\\
Let $K$ be a number field, $\Gamma=\Gal\big(K^{\text{cyc}}/K\big)$, $K_n=\text{ the fixed field of }\Gamma^{p^n}$, $A'n=A'(K_n)$ and $A'\big(K^{\text{cyc}}\big)=\varinjlim_nA'_n$ as before. For $h\geq n\geq 0$ write $G^h_n=\Gamma^{p^n}/\Gamma^{p^h}=\Gal\big(K_h/K_n\big)$. The \emph{kernel and cokernel of capitulation} associated with the extension $K_h/K_n$ are defined by the exact sequence
\begin{center}
\begin{tikzcd}[column sep=scriptsize]
0\arrow[r] & {\rm Cap}(K_h/K_n) \arrow[r] & A'_n \arrow[r, , "j_{n,h}"] & (A'_h)^{G_n^h} \arrow[r] & {\rm Cocap}(K_h/K_n)\arrow[r] &0
\end{tikzcd}
\end{center}
One defines analogously ${\rm Cap}\big(K^{\text{cyc}}/K_n\big)$ and  ${\rm Cocap}\big(K^{\text{cyc}}/K_n\big)$. The asymptotic links between the (co)kernels in the capitulation exact sequence can be summarized in the following
%%
%%proposition
%%
\begin{prop}\label{p:10-5} Assume the Kuz'min-Gross conjecture for all $n$ large\\ enough. Then:
\begin{description}
\item[(1)] The kernel ${\rm Cap}\big(K^{\text{cyc}}/K_n\big)$ stabilize w.r.t. norm maps between ideal classes and become isomorphic to $X'\big(K^{\text{cyc}}\big)^0$ for $n\gg 0$;
\item[ (2)] The kernel $\Psi(K_n)$ of the natural maps $X'\big(k^{\text{cyc}}\big)_{\Gamma^{p^n}}\rightarrow A'_n $ stabilize w.r.t. extension maps and become isomorphic to $\Psi\big(K^{\text{cyc}}\big)$ for $n\gg 0$. More precisely, for $h\geq n\gg0$, the map $v_{h,n}:\ X'\big(K^{\text{cyc}}\big)_{\Gamma^{p^n}}\rightarrow X'\big(K^{\text{cyc}}\big)_{\Gamma^{p^h}}$ of definition~\ref{d:10-5} induces the isomorphisms $\Psi_n\cong\Psi_h\cong\Psi\big(K^{\text{cyc}}$\big);
\item[(3)] The cokernels ${\rm Cocap}\big(K^{\text{cyc}}/K_n\big)$ also stabilize and become isomorphic to $\Psi\big(K^{\text{cyc}}\big)$. More precisely, the isomorphism is 
\begin{equation*}
{\rm Cocap}\big(K^{\text{cyc}}/K_n\big)\cong\Psi\big(K^{\text{cyc}}\big)_{\Gamma^{p^n}}\cong \Psi\big(K^{\text{cyc}}\big)\text{ for } n\gg 0\, .
\end{equation*}.
\end{description}
\end{prop}
%%
%%proof
%%
\begin{proof}  Property {\bfseries (1)} is a classical result (\cite{rgre1}, \cite{kiwa2}, \cite{ikuz1}). For property {\bfseries (2)} and {\bfseries (3)}, see \cite{fl_mo_ng1}, lemma 1.3 and theorem 1.4.
\end{proof}
All the above assertions are asymptotic in nature. The proof of our technical theorem~\ref{t:10-5} will require the following more precise
%%
%%lemma
%%
\begin{lm}\label{l:10-5} For any $n\geq 0$, the kernel  $\Phi_n=\Phi(K_n)$  of the natural map $\Psi_n=\Psi(K_n)\rightarrow \Psi(K^{\text{cyc}})$ is described by two exact sequences
\begin{gather*}
0\rightarrow \Phi_n\rightarrow \Psi_n\rightarrow \Psi\big(K^{\text{cyc}}\big)^{\Gamma^{p^n}}\rightarrow 0\\
\text{and}\\
0\rightarrow\Phi_n\rightarrow X'\big(K^{\text{cyc}}\big)^0_{\Gamma^{p^n}}\rightarrow {\rm Cap}\big(K^{\text{cyc}}/K_n\big)\rightarrow \dots
\end{gather*}
For $n\gg0$ we have $\Phi_n=(0)$.
\end{lm}
%%
%%proof
%%
\begin{proof}
Let us come back to the decomposition exact sequence of class field theory introduced in the proof of theorem~\ref{t:5-5} and modify slightly the notations: $\overline{K^*_{n,v}}=$pro-p-completion of $K^*_{n,v}$; $\widehat{K^*_{n,v}}=$the subgroup of universal norms, such that $\overline{K^*_{n,v}}/\widehat{K^*_{n,v}}\cong \Gamma^{p^n}$ by local class field theory; $\overline{U'_n}=\overline{U'_{L_n}};\widehat{U'_n}=$ the kernel of the diagonal map $\overline{U'_n}\rightarrow \bigoplus_{v\divi p}\overline{K^*_{n,v}}/\widehat{K^*_{n,v}}$, which is the subgroup of $\overline{U'_n}$ consisting of elements which are universal norms locally everywhere. Because of the product formula, the image of the diagonal map is actually contained in the kernel $\widetilde{\bigoplus_{v\divi p}} \overline{K^*_{n,v}}$ of the linear form ``sum of coordonates'', and we have the so called Sinnott exact sequence of $G_n$-modules (\cite{fe_gr1}, appendix)
\begin{equation}\label{e:10-5}\tag{$\text{Sinn}_n$}
0\rightarrow \overline{U'_n}/\widehat{U'_n}\rightarrow M_n=\widetilde{\bigoplus_{v\divi p}} \overline{K^*_{n,v}}/\widehat{K^*_{n,v}}\rightarrow \Psi_n\rightarrow 0
\end{equation} 
where the middle term is an induced module.\medskip

For $h\geq n$, the natural maps between the levels $h$ and $n$ allow to compare the sequences ($\text{Sinn}_n$) and ($\text{Sinn}_h$). Denote $G_n^h=\Gal(K_h/K_n)$ and let $G_{n,v}^h$ be the corresponding local Galois group, more precisely $G_{n,v}^{h}=\Gal\big(K_{h,w}/K_{n,v}\big)$ for an arbitrarily fixed choice of $w\divi v$. Il is important to note that by definition of $\widehat{K^*_{h,w}}$, $G_{n,w}^h$ acts trivially on $\overline{K^*_{h,w}}/\widehat{K^*_{h,w}}\cong\Gamma^{p^h}$. It can be shown ``by hand'' (see \cite{tngu3}, proposition~2.2) or, more easily, by using cohomological machinery  for universal norm (local here) as in \cite{kkat1}, (1b), p.553, that $\widehat{K^*_{n,v}}\cong (\widehat{K^*_{h,w}})^{G_{n,v}^h}$, thus by the cohomological triviality of the local universal norms, $\overline{K^*_{n,w}}/\widehat{K^*_{n,w}} \cong \big(\overline{K^*_{h,w}}/\widehat{K^*_{h,w}}\big)^{G_{n,v}^h} \cong \overline{K^*_{h,w}}/\widehat{K^*_{h,w}}$. We get from Shapiro's lemma that $M_n\cong ({M_h})^{G_n^h}$ (and $M_n\cong M_h$ for $h\geq n\gg 0$).  The snake lemma then shows that the map $\Psi_n\rightarrow{\Psi_h}^{G_n^h}$ is surjective and its kernel is isomorphic to the cokernel of the injective map $\overline{U'_n}/\widehat{U'_n}\rightarrow \big(\overline{U'_h} /\widehat{U'_h}\big)^{G_n^h}$. Letting $h$ grow to infinity, we have the exact sequences
\begin{equation*}
0\rightarrow \Phi_n\rightarrow \Psi_n\rightarrow \Psi\big(K^{\text{cyc}}\big)^{\Gamma^{p^n}}\rightarrow 0
\end{equation*}
(which is the first sequence of lemma~\ref{l:10-5}) and
\begin{equation*}
0\rightarrow \overline{U'_n} /\widehat{U'_n}\rightarrow \varinjlim_h\big( \overline{U'_h} /\widehat{U'_h}\big)^{G_n^h}\rightarrow \Phi_n\rightarrow 0\, .
\end{equation*}
To go further, let us introduce the subgroup $\widetilde{U'_n}$ of $\overline{U'_n}$ consisting of global universal norms, which is described  more precisely by co-descent as $\widetilde{U'_n}\cong (\varprojlim_h\overline{U'_h})_{\Gamma^{p^n}}$. It is known (see e.g. \cite{ikuz1}) that
\begin{itemize}
\item the $A$-module $\varprojlim_h\overline{U'_h}$ is isomorphic to $\Zfield_p(1)\bigoplus \Lambda^{1+r_2}$ or $\Lambda^{1+r_2}$ according as $K$ contains or not a primitive $p$-th root of unity, hence $\widetilde{U'}_n$ is $G_n$-cohomologically trivial;
\item the quotient $\widehat{U'_n}/\widetilde{U'_n}$ is canonically isomorphic to $X'\big(K^{\text{cyc}}\big)^{\Gamma^{p^n}}$ (this can be considered as a kind of obstruction to a Hasse principle in the present situation).
\end{itemize}
Put $V_h=\overline{U'_n}/\widehat{U'_h}$, $W_h=\widehat{U'_h}/\widetilde{U'_h}$, $X'_{\text{cyc}}=X'\big(K^{\text{cyc}}\big)$. Taking the $G_n^h$-cohomolo\-gy of the tautological exact sequence 
\begin{equation*}
0\rightarrow \big(X'_{\text{cyc}}\big)^{\Gamma^{p^h}}\rightarrow W_h\rightarrow V_h\rightarrow 0\, , 
\end{equation*}
we get
\begin{multline*}
0 \rightarrow \big(X'_{\text{cyc}}\big)^{\Gamma^{p^n}}\rightarrow (W_h)^{G_n^h}\rightarrow (V_n)^{G_n^h}\rightarrow H^1\big( G_{G^h_n}, (X'_{\text{cyc}})^{\Gamma^{p^h}}\big)\rightarrow\\ \rightarrow H^1(G_n^h, W_h)
\rightarrow H^1(G_n^h,V_h)\rightarrow \dots
\end{multline*}
By Kato's result on (global here) universal norms (\cite{kkat1}, \emph{op. cit.}), $\widetilde{U'_n}\cong (\widetilde{U'_h})^{G_n^h}$ hence  $W_n\cong (W_h)^{G_n^h}$ and our exact sequence becomes
\begin{multline*}
0 \rightarrow V_h \rightarrow (V_h)^{G_n^h}\rightarrow H^1\big( G_{G^h_n}, (X'_{\text{cyc}})^{\Gamma^{p^h}}\big)\rightarrow\\ \rightarrow H^1(G_n^h, W_h)
\rightarrow H^1(G_n^h,V_h)\rightarrow \dots
\end{multline*}
or taking $\varinjlim$ and recalling that $\big(X'_{\text{cyc}}\big)^{\Gamma^{p^h}}=\big((X'_{\text{cyc}})^0\big)^{\Gamma^{p^h}}$ by the Kuz'min-Gross conjecture:
\begin{multline}\label{e:10-10}\tag{\bfseries A}
0\rightarrow \Phi_n\rightarrow H^1\big(\Gamma^{p^n}, (X'_{\text{cyc}})^0\big) = {(X'_{\text{cyc}})^0}_{\Gamma^{p^n}}\rightarrow \\ \rightarrow H^1(\Gamma^{p^n}, \varinjlim_h W_h)\rightarrow H^1(\Gamma^{p^n},\varinjlim_h V_h)\dots
\end{multline}
Taking into account the equality $H^1(G_n^h, W_h)=H^1(G_n^h, \overline{U'_h})$ because $ \overline{U'_h}$ is cohomologically trivial, we get from the above \eqref{e:10-10}, when $h$ grows to infinity, the exact sequence
\begin{equation*}
0\rightarrow \Phi_n\rightarrow {(X'_{\text{cyc}})^0}_{\Gamma^{p^n}}\rightarrow H^1(\Gamma^{p^n},\varinjlim_h \overline{U'_h})\rightarrow\dots
\end{equation*}
but the term $H^1(\Gamma^{p^n},\varinjlim_h \overline{U'_h})$ is known to be isomorphic to ${\rm Cap}(K^{\text{cyc}}/K_n)$ (e.g. \cite{ng_le1}, theorem~1.1), whence the second exact sequence of lemma~\ref{l:10-5}.\medskip

Let now $n$ grow to infinity. The last term in \eqref{e:10-10} becomes null because $G_n^h$ acts trivially on the $\Zfield_p$-free module $M_h$ (see above). The comparison  with \eqref{e:10-5} and a quick diagram chase immediately show the nullity of $H^1(G_n^h,V_h)$ hence the exact sequence
\begin{equation*}
0\rightarrow \Phi_n\rightarrow  (X'_{\text{cyc}})^0_{\Gamma^{p^n}}\rightarrow {\rm Cap}(K^{\text{cyc}}/K_n) \rightarrow 0\, .
\end{equation*}
But for $n \gg 0$, the rightmost map is an isomorphism so $\Phi_n=(0)$ (this also follows from proposition~\ref{p:10-5} {\bfseries (2)}).
\end{proof}
%%
%%proof
%%
\begin{proof}[\emph{\bfseries Proof of theorem~\ref{t:10-5}:}] Going back to the notations of the theorem (in particular, the base field is $k$), we must compute $\charac_{\Lambda}\big(X'^{(2)}\big)_{\Gamma}/\charac_{\Lambda}\big(X'^{(2)}\big)^{\Gamma}$ in order to use lemma~\ref{l:5-10}.\medskip

Write 
\begin{gather*}
K_{m,n}=F_m\cdot k_n, \ \Gamma_m=\Gal\big(F^{\text{cyc}}_m/F_m\big)\cong \Gamma =\Gal\big(k^{\text{cyc}}/k\big), \\ \Gamma_{m,n}=\Gal\big(F^{\text{cyc}}_m/K_{m,n}\big),\
 G_{m,n}=\Gal(K_{m,n}/F_m)\, .
\end{gather*}
 Recall from definition~\ref{d:10-5} that for all $m,n\geq 0$ 
\begin{gather*}
\Gamma_m=\Gal\big(F^{\text{cyc}}_m/F_{m}\big)\cong\Gamma=\Gal\big(k^{\text{cyc}}/k\big),\\
 \Gamma_{m,n} =\Gal\big(F^{\text{cyc}}_m/K_{m,n}\big)\cong \Gamma^{p^n},\ G_{m,n}=\Gal(K_{m,n}/F_m)
\end{gather*} (beware of a possible confusion with $G_n^h$ above). Again, it is recommanded to draw a Galois diagram.
%%
%%paragraph
%%
\paragraph{(i) Computation of $\big(X'^{(2)})^{\Gamma}$.} Taking $G_{ m,n}$-invariants in the first half of the cap-cocap exact sequence of proposition~\ref{p:10-5} relative to $F_m^{\text{cyc}}/K_{m,n}$, we get
\begin{equation*}
0\rightarrow {\rm Cap}\big(F_m^{\text{cyc}}/K_{m,n}\big)^{G_{m,n}} \rightarrow A'(K_{m,n})^{G_{m,n}}\rightarrow A'(F_m^{\text{cyc}})^{\Gamma}\rightarrow \dots
\end{equation*}
Let us compute $\varprojlim_{m,n}$ of the three term in this sequence. When taking  
\begin{equation*}
\varprojlim_{m,n}A'(F_m^{\text{cyc}})^{\Gamma}
\end{equation*}
 we pass first, for $m$ fixed, through the inverse limit $\varprojlim_{n}A'(F_m^{\text{cyc}})^{\Gamma}$ w.r.t. $p$-power-th maps, which is null because $A'(F_m^{\text{cyc}})^{\Gamma}$ is finite. 

As for $\displaystyle \varprojlim_{m,n}{\rm Cap}\big(F_m^{\text{cyc}}/K_{m,n}\big)^{G_{m,n}}$, we already know that, for $m$ fixed and $n$ large ${\rm Cap}\big(F_m^{\text{cyc}}/K_{m,n}\big)\cong X'\big(F^{\text{cyc}}_m\big)^0$, so that ultimately 
\begin{equation*}
 \big(X'^{(2)}\big)^{\Gamma}\cong \varprojlim_{m} \big(X'\big(F^{\text{cyc}}_m\big)^0\big)^{\Gamma}
 \end{equation*} 
 which is finite by hypothesis. Thus it remains to study the finiteness of $\big(X'^{(2)}\big)_{\Gamma}$.
%%
%%paragraph
%%
\paragraph{(ii) Computation of $\big(X'^{(2)}\big)_{\Gamma}$.} By definition 
\begin{equation*}
\big(X'^{(2)}\big)_{\Gamma}=\varprojlim_{m,n} A'(K_{m,n})_{G_{m,n}}\, .
\end{equation*} 
First and foremost, we must explain the relationship between 
\begin{equation*}
 \varprojlim_{m,n}\big(X'\big(F^{\text{cyc}}_m\big)_{\Gamma_{m,n}}\big)_{G_{m,n}}\text{ and } \varprojlim_{m}X'\big(F^{\text{cyc}}_m\big)_{\Gamma_{m}}\, .
\end{equation*}
Starting from the exact sequence
\begin{equation*}
0\rightarrow \Psi(K_{m,n}) \rightarrow X'\big( F^{\text{cyc}}_m\big)_{\Gamma_{m,n}}\rightarrow A'(K_{m,n})\rightarrow 0
\end{equation*}
of definition~\ref{d:10-5}, apply the snake lemma for the action of $G_{m,n}$ to derive the exact sequence of invariants-coinvariants 
\begin{multline*}
\dots \rightarrow A'(K_{m,n})^{G_{m,n}}\rightarrow \Psi(K_{m,n})_{G_{m,n}}\rightarrow \big(X'\big(F^{\text{cyc}}_m\big)_{\Gamma_{m,n}}\big)_{G_{m,n}}\rightarrow \\
\rightarrow A'(K_{m,n})_{G_{m,n}}\rightarrow 0\, .
\end{multline*}
Let us take $\varprojlim_{m,n}$ along a cofinal system of lexicographically ordered indices by considering the following diagram where , for $(m',n')\geq (m,n)$, the vertical connecting maps , generically denoted $N^{m',n'}_{m,n}$, are induced by norms. We obtain the following commutative diagram ({\bfseries B}):

\hspace{-60pt}
\begin{tikzpicture}[>=angle 90,scale=1,text height=1.5ex, text depth=0.25ex, x=0.5cm]
%%First place the nodes
\node (k1) at (0,0) {$A'(K_{m',n'})^{G_{m',n'}}$};
\node (k0) [left=of  k1] {$\Psi(K_{m',n'})_{G_{m',n'}}$};
\node (k-1) [left=of k0]  {${}$};
\node (a-1) [below=of k-1] {${}$};
\node (a1) [below=of k0] {$A'(K_{m,n})^{G_{m,n}}$};
\node (a2) [below=of k1] {$\Psi(K_{m,n})_{G_{m,n}} $};
\node (k2) [below=of a1] {$\big(X'(F^{\text{cyc}}_{m'})_{\Gamma_{m',n'}}\big)_{G_{m',n'}}$};
\node (a3) [below=of k2] {$ \big(X'(F^{\text{cyc}}_{m})_{\Gamma_{m,n}}\big)_{G_{m,n}}$};
\node (b1) [below=of a2] {$A'(K_{m',n'})_{G_{m',n'}}'$};
\node (b2) [right=of b1] {$0$};
\node (c1) [below=of b1] {$A'(K_{m,n})_{G_{m,n}}$};
\node (c2) [below=of b2] {$0$};
%%Draw the curvy black arrows
\draw[->,black]
(k1) edge[out=0,in=180,black]  (k2)
(a2) edge[out=0,in=180,black]  (a3);
%%Draw the black arrows
\draw[->, black, thick]
(k0) edge (k1)
(k2) edge (b1)
(c1) edge (c2)
(k0) edge (a1)
(k1) edge (a2)
(k2) edge (a3)
(a3) edge (c1)
(b1) edge (c1)
(a1) edge (a2)
(b1) edge (b2);
\draw[->,black,dashed]
(k-1) edge (k0)
(a-1) edge (a1);
\end{tikzpicture}

The connecting maps $N_{m,n}^{m',n'}$ between the extreme terms are naturally induced by the norm maps from $K_{m',n'}$ to $K_{m,n}$ which factorize as $N_{m,n}^{m',n'}=N_{m,n}^{m',n}\circ N_{m',n}^{m',n'}$ by going through $K_{m',n}$\, . \\
In the middle column, we may identify $\big(X'(F^{\text{cyc}}_{m})_{\Gamma_{m,n}}\big)_{G_{m,n}}$ with $X'\big(F^{\text{cyc}}_m\big)_{\Gamma_m}$, but then we must compare $\displaystyle\varprojlim _m X'\big(F^{\text{cyc}}_m\big)_{\Gamma_m}$ and  $\displaystyle\varprojlim _{m,n} \big(X'\big(F^{\text{cyc}}_m\big)_{\Gamma_{m,n}}\big)_{G_{m,n}}$ using the factorization of $N_{m,n}^{m',n'}$: more precisely $N_{m',n}^{m',n'}$ is induced by $\Id$ of $X'\big(F^{\text{cyc}}_{m'}\big)$ because the elements of $X'\big(F^{\text{cyc}}_{m'}\big)$ are already coherent normic systems w.r.t. the index $n$, whereas $N_{m,n}^{m',n}$ is induced by the norm from $X'\big(F^{\text{cyc}}_{m'}\big)$ to $X'\big(F^{\text{cyc}}_{m}\big)$.\\medskip

Thus taking $\displaystyle\varprojlim _{m,n}$ in diagram ({\bfseries B}) gives an exact sequence
\begin{multline*}
\dots\varprojlim _{m,n}A'(K_{m',n'})^{G_{m,n}}\rightarrow \varprojlim _{m,n}\Psi(K_{m,n})_{G_{m,n}} \rightarrow \varprojlim _{m,n} \big(X'(F^{\text{cyc}}_{m})_{\Gamma_{m,n}}\big)_{G_{m,n}}\cong \\
\cong\varprojlim _{m}X'(F^{\text{cyc}}_{m})_{\Gamma}\rightarrow \big(X'^{(2)}\big)_{\Gamma}\rightarrow 0\, .
\end{multline*}
But for $m$ given and  $n$ large $\Psi(K_{m,n})\cong \Psi\big(F^{\text{cyc}}_m\big)$ via the direct limit of the extension maps $\Psi(K_{m,n})\rightarrow \Psi(K_{m,h})$ induced by $\nu_{h,n}=\frac{\gamma^{p^h}-1}{\gamma^{p^n}-1}$ for $h\geq n$ ( see proposition~\ref{p:10-5} {\bfseries (2)}), so that $\Psi(K_{m,n})_{G_{m,n}}$ can be viewed as the co-invariant quotient of $\Psi(F^{\text{cyc}}_{m})$ under the action of $G_{m,n}$, more precisely $\Psi(K_{m,n})_{G_{m,n}}\cong \Psi(F^{\text{cyc}}_{m})/(\gamma^{p^h}-1)$, and the previous exact sequence reads
\begin{equation*}
\dots\big( X'^{(2)}\big)^{\Gamma}\rightarrow \varprojlim _{m}\Psi(F^{\text{cyc}}_{m})_{\Gamma} \rightarrow \varprojlim _{m} X'(F^{\text{cyc}}_{m})_{\Gamma}\rightarrow \big(X'^{(2)}\big)_{\Gamma}\rightarrow 0\, .
\end{equation*}
Because of our hypothesis on $\displaystyle  \varprojlim _{m}\Psi(F^{\text{cyc}}_{m})$, the finiteness of $\big(X'^{(2)}\big)_{\Gamma}$ becomes equivalent to that of $\varprojlim _{m} \big(X'(F^{\text{cyc}}_{m})\big)_{\Gamma}$.\medskip

We must now appeal to lemma~\ref{l:10-5} in order to work at the level $n=0$. Recall the two exact sequences
\begin{gather*}
0\rightarrow\Phi(F_m)\rightarrow X'\big(F^{\text{cyc}}_m\big)^0_{\Gamma}\rightarrow \mathrm{Cap}\big(F^{\text{cyc}}_m/F_m\big) \rightarrow\dots\\
\text{and}\\
0\rightarrow\Phi(F_m)\rightarrow \Psi(F_m)\rightarrow \Psi\big(F^{\text{cyc}}_m\big)^{\Gamma}\rightarrow 0\,
\end{gather*} 
Because $X'\big(F^{\text{cyc}}_m\big)^0_{\Gamma}$ and $\big(X'\big(F^{\text{cyc}}_m\big)^0\big)^{\Gamma}$ have the same order, the finiteness by hypothesis of $\displaystyle \varprojlim_m\big(X'\big(F^{\text{cyc}}_m\big)^0\big)^{\Gamma}$ implies that of  $\displaystyle \varprojlim_m\Phi(F_m)$, so $\displaystyle \varprojlim_m\Psi(F_m)$ becomes pseudo-isomorphic to $\displaystyle \varprojlim_m \Psi\big(F^{\text{cyc}}_m\big)^{\Gamma}$, which is finite by hypothesis. This means, by the very definition of $\Psi(F_m)$ as the kernel of $X'\big(F^{\text{cyc}}_m\big)\rightarrow A'(F_m)$, that $\displaystyle \varprojlim_mX'\big(F^{\text{cyc}}_m\big)_{\Gamma}$ is pseudo-isomorphic to $\displaystyle  \varprojlim_m A'(F_m)$, and shows, as desired, that $\big(X'^{(2)}\big)_{\Gamma}$ is finite if and only if $\displaystyle  \varprojlim_m A'(F_m)$ is finite.\end{proof}
%%
%%section
%%
\section{\large\label{sec15} Undecomposed $p$-primes and density properties.}
In the rest of the paper, our problem will be to provide families of number fields to which theorem~\ref{t:10-5} can apply. In principle, it ``suffices'' to exhibit a non cyclotomic $\Zfield_p$-extension $F_{\infty}/k$ such that $\lambda'(F_{\infty})=\mu'(F_{\infty})=0$. However this is not easy. Let $\Ecal(k)$ be the space of all $\Zfield_p$-extension of $k$ equipped with the Greenberg topology. It is easy to show that the desired extensions $F_{\infty}/k$ form an open  and closed  subspace of $\Ecal(k)$, but the point is that it could be empty. For commodity, let us baptize \emph{obstruction kernels} the three inverse limits $\displaystyle \varprojlim_m\big(X'\big(F^{\text{cyc}}_m\big)^0\big)^{\Gamma}$, $\displaystyle \varprojlim_m \Psi\big(F^{\text{cyc}}_m\big)$ and  $\displaystyle  \varprojlim_m A'(F_m)$, whose simultaneous finiteness allows to conclude that $K^{(2)}/k$ is green. Let us write (I), (II), (III) for these respective properties  and study their interdependence. Using lemma~\ref{l:10-5}, one can readily show that (II)+(III) imply (I), (III)+(I) imply the finiteness of  $\displaystyle \varprojlim_m \Psi(F_m)$, and (I)+(II) imply that this finiteness is equivalent to (III). Note the absence of a direct implication between (I) and (II) alone.\medskip

There are two possible ways to build a $\Zfield_p^2$-extension $K^{(2)}=k^{\text{cyc}}\cdot F_m$ such as in theorem~\ref{t:10-5}:
\begin{description}
\item[(i)] Given a base field $k$, build a special independent $\Zfield_p$-extension $F_{\infty}/k$ s.t. $K^{(2)}$ is green;
\item[(ii)] Construct an adequate base field $k$ s.t for any independent $\Zfield_p$-extension $F_{\infty}/k$, the compositum  $K^{(2)}$ of $k^{\text{cyc}}$ and $F_{\infty}$ is at least, in a precise sense, ``partially green'' (see \S\ref{sec25}).
\end{description}
We undertake the approach {\bfseries (i)} in this section. The ``undecomposed $p$-primes''  in the headline mean that that we are mainly interested in base fields $k$ such that no prime of $k$ above $p$ can decompose in $k^{\text{cyc}}$ (but $p$ can decompose in $k$). Let us first collect (without any pretention to exhaustivity) some miscellaneous known results depending on the decomposition of the rational primes $p$ in $k/\Qfield$:
%%
%%theorem
%%
\begin{prop}\label{p:15-1} \emph{\bfseries (1) Suppose that $p$ does not decompose in $k$} and $A(k)$ is null. Then $X(\tilde{k})$ is null,\\
\emph{\bfseries (2)} Suppose that $p$ does not decompose in $k$ and its class generates $A(k)$. Then $k$ satisfies \emph{GGC},\\
\emph{\bfseries (3)} Let $k$ be the abelian extension of an imaginary quadratic field $\kappa$ such that $p$ splits in $\kappa$ but does not decompose further in $k$. Suppose that $A(k)$ is cyclic, generated  by the classes of the two ideals above $p$. Then $k$ satisfies \emph{GGC}.
\end{prop}
%%
%%proof
%%
\begin{proof} Statement {\bfseries (1)} is an early result of \cite{kiwa2}. Let us show the following refinement from \cite{lwas1}, chap. 10, theorem~10.4: If $L/k$ is a (normal) $p$-extension such that at most one prime of $k$ ramifies in $L$, then the nullity of $A(k)$ implies that of $A(L)$. Since any $\Zfield_p^d$-extension $K^{(d)}$ of $k$ is unramified outside $(p)$, the non-decomposition of $p$ in $k$ allows to apply the theorem to any extension $L/k$ contained in $K^{(d)}$, and conclude that $X(K^{(d)})=0$. As for {\bfseries (2)} (resp. {\bfseries (3)}), see  \cite{skle1}, theorem~4.6 (resp. \cite{jmin1}, proposition 3.A).
\end{proof}
Actually the reminders in proposition~\ref{p:15-1} can be  seen to belong to the general setting of theorem~\ref{t:15-10} below. Rather unexpectedly, the desired existence of a special $\Zfield_p$-extension $F_{\infty}/k$ such that $X'(F_{\infty})$ is finite can be viewed, in the the setting of theorem~\ref{t:10-5}, as a kind of converse to a ``density'' result due to T. Kataoka in the topological space $\Ecal(k)$ just  introduced above. More precisely for a number field $k$, let $\Ecal_{\text{ram}}(k)$ (resp. $\Ecal_{\text{ns}}(k)$) be the set of $\Zfield_p$-extension $F_{\infty}/k$ in which every $p$-prime of $k$ is ramified (rep. undecomposed). The ``non-decomposition'' headline means that our main subspace of interest will be $\Ecal_{\text{ns}}(k)$, which is open and closed in $\Ecal(k)$ (\cite{tkat1}, lemma~5.1). To ensure that $\Ecal_{\text{ns}}(k)\neq\emptyset$, we can  suppose that $k^{\text{cyc}}\in \Ecal_{ns}(k)$, which happens e.g. when $p\ndivi [k/\Qfield]$ or $p$ unramified in $k$ (\cite{tkat1}, remark~2.5). Denote by $s(F_{\infty}/k)$ the $\Zfield_p$-rank of $X(F_{\infty})_{G(F_{\infty}/k)}$ and by $s(k)$ the minimum of $s(F_{\infty}/k)$ in $\Ecal_{\text{ram}}(k)$. From theorems~1.3 and 5.3 of \cite{tkat1} we extract the following
%%
%%theorem
%%
\begin{thm}\label{t:15-5} Suppose that the number field $k$ satisfies \emph{GGC} and the conjecture of Kuz'min-Gross. Then the subset of $\Ecal_{ns}(k)$ consisting of the $\Zfield_p$-extensions $K/k$ such that $\mu(K/k)=0$ and $\lambda(K/k)=s(k)$ contains an open dense subset of $\Ecal_{ns}(k)$. Moreover, $s(k)=0$ if $p$ does not decompose in $k$ (exemple~4.4.2). If $p$ splits  completely in $k$, suppose in addition that $k$ is CM and satifies Leopoldt's conjecture; then $s(k)=[k:\Qfield]/2$ (exemple~4.4.1).
\end{thm}
%%
%%theorem
%%
\begin{thm}\label{t:15-10}  Let $k$ be a number field such that $k^{\text{cyc}}\in\Ecal_{ns}(k) $. Then, modulo the Kuz'min-Gross conjecture:
\begin{description}
\item[(a)] If $p$ is not decomposed in $k$, \emph{GGC} holds if and only if $k$ admits an independent $\Zfield_p$-extension $F_{\infty}=\cup F_m$ such that $\lambda'(F_{\infty})=\mu'(F_{\infty})=0$ and that $\displaystyle\varprojlim_{m}\Psi\big(F^{\text{cyc}}_m\big)$ is finite.
\item[(b)] If $p$ splits completely in $k$, suppose in addition that $k$ is CM and satisfies Leopoldt's conjecture. Then the same conclusion as in (a) remains valid.
\end{description}
\end{thm}
%%
%%proof
%%
\begin{proof}It suffices to show the ``if'' part in cases {\bfseries (a)} and {\bfseries (b)} of theorem~\ref{t:15-10}. As for {\bfseries (a)}, this just means that (II)+(III) imply (I), as we have seen at the beginning of \S\ref{sec15}. As for {\bfseries (b)}, it is classically known (\cite{ne_sc_wi1}, proposition~114.5 and 11.4.7) that
\begin{equation*}
\lambda(k^{\text{cyc}})-\lambda'(k^{\text{cyc}})=\#S_p(k^{\text{cyc}})-\#S_p(k^{\text{cyc},+})
\end{equation*}
for a CM-fields, where $S_p(.)$ denotes the set of $p$-places. But our ramification hypothesis implies that, modulo Leopoldt's conjecture 
\begin{equation*}
\#S_p(k)-\#S_p(k^+)=\frac{[k:\Qfield]}{2}=\#S_p(k^{\text{cyc}})-\#S_p(k^{\text{cyc}, +})\, .\hspace{15mm}\qed
\end{equation*}
\renewcommand{\qed}{}\end{proof}
Thus, in the setting of theorem~\ref{t:10-5}, theorems~\ref{t:15-5} and \ref{t:15-10} assert, roughly speaking, that \emph{GGC} is equivalent to (I)+(II)+(III).
%%
%%exemples
%%
\begin{exs}\label{ex:15-5} It is easy to provide examples with trivial (non only finite) obstruction  kernels $\Psi(.)$ because the nullity of $\Psi(k^{\text{cyc}})$ propagates to all the $\Psi(F^{\text{cyc}}_m)$'s (see proposition~\ref{p:20-10} below). The triviallity of $\Psi(k^{\text{cyc}})$ is an asymptotic property, but which can be detected at finite levels. For example, the nullity occurs if the prime $p$ does not decompose in $k^{\text{cyc}}/\Qfield$, see proposition~\ref{p:15-1} {\bfseries (1)} above. However this is not the only possibility:
\begin{description}
\item[(a)] Assuming the Kuz'min-Gross conjecture at all levels $n\gg 0$, the nullity of $A'(k^{\text{cyc}})$ (which is Greenberg's classical conjecture when the base field $k$ is totally real ) is equivalent to the nullity of $\Psi(k_n)$ together with the capitulation of $A'(k_n)$ in $k^{\text{cyc}}$ for $n\gg 0$ (\cite{fl_mo_ng1}, corollary~1.8).
\item[(b)] In general, if the vanishing of $\Psi(k_N)$ occurs at a certain level $N$, then it propagates to all levels, and the map 
\begin{equation*}
A'(k_n)\longrightarrow A'(k^{\text{cyc}})^{\Gamma^{p^n}}
\end{equation*}
is surjective  for $n\geq N$ (\cite{fl_mo_ng1}, corollary~1.6).
\end{description}
As it is well known, the analogous problem for $\big(X'\big(F^{\text{cyc}}_m\big)^0\big)^{\Gamma}$ can be much harder. But recall that  (II)+(III) imply (I). For a more systematic study, see \S\ref{sec20} below.
\end{exs}
%%
%%section
%%
\section{\large Nullity of the obstructions kernels \label{sec20}}
The construction {\bfseries (ii)} evoked at the beginning of \S\ref{sec15}, the auxiliary $\Zfield_p$-exten\-sion $F_{\infty}/k$ being arbitrary, the condition imposed on (resp. the results obtained from) the control of the obstruction kernels will naturally be stronger (resp. weaker). Let us first look at the simultaneous nullity of $\Psi\big(K{\text{cyc}}\big)$ and $X'\big(K^{\text{cyc}}\big)^0$ for a given number field $K$. Examples of fields $K$ with trivial $\Psi\big(K^{\text{cyc}}\big)$ or $X'\big(K^{\text{cyc}}\big)^0$ (separately) can be found in the literature, but our problem from here on will be twofold:
\begin{itemize}
\item to construct families of fields $K$ for which both$\Psi\big(K^{\text{cyc}}\big)$ and $X'\big(K^{\text{cyc}}\big)^0$ are trivial;
\item to study the propagation of the vanishing  conditions from $K$ to $F^{\text{cyc}}$. The point is that, in the non semi-simple case, the modules $X'(.)$ do not behave obediently under descent. To circumvent this difficulty, we must appeal to ``reflection relations'' (``Spielgelung''), which will require beforehand.
\end{itemize}
%%
%%paragraph
%%
\paragraph{A survey of the Bertrandias-Payan module.}(see \cite{fl_mo_ng1}, \cite{ng_ni1})

For a number field $K$ containing $\mu_p$ (with no other assumption), ``Spielgelung'' consists in a combination of isomorphism  (class field theory) and duality (Kummer theory), so it is best expressed above $K(\mu_{p^{\infty}}):=K_{\infty}$, in the framework of Iwasawa theory. For the following reminders, we refer mainly to \cite{fl_mo_ng1}, \S\S2-3; a convenient detailed summary can be found in \cite{ng_ni1} \S6. Let us fix some notations, most of them coming from \cite{kiwa2}: let $M_{\infty}$ be the maximal abelian pro-$p$-extension of $K_{\infty}$ unramified outside $p$, so that  $Y(K_{\infty})=\Gal(M_{\infty}/K_{\infty})$, and let $T_{\infty}$ be the fixed field of $t_{\Lambda}Y(K_{\infty})$, the $\Lambda$-torsion of $Y(K_{\infty})$, with $\Lambda=\Zfield_p[[T]]$. The field $M_{\infty}$ contains three remarquable subfields, $L'_{\infty}=$the maximal abelian unramified pro-$p$-extension of $K_{\infty}$ such that all the $p$-places of $K_{\infty}$ are totally split in $L'_{\infty}$; $N'_{\infty}=$the field obtained by adding to $K_{\infty}$ all $p$-primary roots of all $p$-units of $K_{\infty}$; and $K_{\infty}^{\text{BP}}=$the so-called field of Bertrandias-Payan (\cite{fl_mo_ng1}, \S2). Although the Bertrandias-Payan fields can be introduced at finite levels in the setting of the Galois embedding problem (as in \cite{tngu1}), we shall give here a definition together with some properties directly at infinite level (as in \cite{tngu2}). It is again recommanded to draw a Galois diagram;
%%
%%subparagraph
%%
\subparagraph{(P1)} The natural surjection $Y(K_{\infty})\rightarrow X'(K_{\infty})$ induces a map 
\begin{equation*}
t_{\Lambda} Y(K_{\infty})\rightarrow X'(K_{\infty})
\end{equation*} 
whose kernel and cokernel can be described using class field theory and Poitou-Tate duality. With obvious notations for local objects $(.)_{v}$, it is known that $t_{\Lambda}Y(K_{\infty})$ contains $W(K_{\infty}):= \bigoplus_{v\divi p}\big(\text{Ind}_{\Lambda_{v}}^{\Lambda}\Zfield_p(1)\big)/\Zfield_p(1)$. Denoting  by $K_{\infty}^{\text{BP}}$ the subfield of $M_{\infty}$ fixed by $W(K_{\infty})$ and calling $\text{BP}(K_{\infty}):=\Gal\big(K_{\infty}^{\text{BP}}/K_{\infty}\big)$ the \emph{Bertandias-Payan module} over $K_{\infty}$, one has by definition a canonical exact sequence of $\Lambda$-modules
\begin{equation*}
0 \rightarrow W(K_{\infty}) \rightarrow t_{\Lambda}Y(K_{\infty}) \rightarrow t_{\Lambda}\text{BP}(K_{\infty})=\Gal(K_{\infty}^{\text{BP}}/T_{\infty})\rightarrow 0
\end{equation*}
%%
%%subparagraph
%%
\subparagraph{(P2)} The relationship with the obstruction kernel $\Psi(K_{\infty})$ appears when going down from $M_{\infty}$ to $T_{\infty}$ through $K_{\infty}^{\text{BP}}$ and $N'_{\infty}$. Get back to the Sinnott exact sequence in lemma~\ref{l:10-5}, where $\widehat{U'_n}$ is the subgroup of $\overline{U'_n}$ consisting of elements which are universal norms locally everywhere:
\begin{equation}\tag{\ref{e:10-5}}
0\rightarrow \overline{U'_n}/\widehat{U'_n} \rightarrow \widetilde{\bigoplus_{v\divi p}} \overline{K^*_{n,v}}/\widehat{K^*_{n,v}}\rightarrow \Psi_n\rightarrow 0
\end{equation}
Tensoring with $\Qfield_p/\Zfield_p$ and taking inductive limits on $n$ we get , with obvious notations:
\begin{equation*}
0\rightarrow \Psi(K_{\infty})\rightarrow \varinjlim \big(\overline{U'}_n/\widehat{U'}_n\big)\otimes\Qfield_p/\Zfield_p\rightarrow \varinjlim \big(\widetilde{\bigoplus_{v\divi p}} \overline{K^*_{n,v}}/\widehat{K^*_{n,v}}\big)\otimes\Qfield_p/\Zfield_p\big)\rightarrow 0
\end{equation*}
(for details see \cite{fl_mo_ng1}, p. 866). Under the Kuz'min-Gross conjecture, we know from \cite{mkol2}, \cite{ikuz1} that $\Gal(T_{\infty}/K_{\infty})=f_{\Lambda}Y(K_{\infty})=Y(K_{\infty})/t_{\Lambda}Y(K_{\infty})$
is the Kummer dual of $\Gal\big({L'}_{\infty}\cap {N'}_{\infty}/{L'}_{\infty}\cap T_{\infty}\big)\cong \Hom(\Psi\big(K_{\infty}),\mu_p^{\infty}\big)$ (\cite{fl_mo_ng1}, theorem~2.5 and 3.2). Note that the Kummer dual of the previous sequence is just the exact sequence of Galois groups
\begin{equation*}
0\rightarrow  W(K_{\infty})\rightarrow \Gal({N'}_{\infty}/T_{\infty})\rightarrow \Gal\big(K_{\infty}^{\text{BP}}\cap N'_{\infty}/T_{\infty}\big)\rightarrow 0\, .
\end{equation*}
The key point is that together with \cite{ng_ni1}, \S6, this gives an ``almost'' direct sum decomposition of the module $t_{\Lambda}Y(K_{\infty})$:
\begin{multline*}
0\rightarrow  W(K_{\infty})\bigoplus t_{\Lambda}\text{BP}(K_{\infty})\rightarrow t_{\Lambda}Y(K_{\infty})\rightarrow  \\
\rightarrow G\big(K_{\infty}^{\text{BP}}\cap N'_{\infty}/T_{\infty}\big)\cong \Hom\big(\Psi(K_{\infty}),\mu_{p^{\infty}}\big)\rightarrow 0
\end{multline*}
The last term is finite. Its nullity is equivalent to $K^{\text{BP}}_{\infty}\cap N'_{\infty}=T_{\infty}$, in which case $\Gal(N'_{\infty}/T_{\infty})\cong W(K_{\infty})$ and $t_{\Lambda}\text{BP}(K_{\infty})\cong \Gal(M_{\infty}/N'_{\infty})$.
%%
%%subparagraph
%%
\subparagraph{(P3)} Independently from any consideration on the Bertrandias-Payan mo\-dule, recall that, after the structure theorem for $\Lambda$-torsion free modules, 
\begin{equation*}
f_{\Lambda} Y(K_{\infty}) =Y(K_{\infty})/t_{\Lambda}Y(K_{\infty})
\end{equation*}
 injects into a free $\Lambda$-module, say with finite cokernel $V_{\infty}$ (note that this cokernel is actually canonical; see \cite{ujan1}). In the kummerian situation, $X'(K_{\infty})^0$ is the Kummer dual of $V_{\infty}$ (\cite{kiwa2}, \cite{ikuz1}), so that $X'(K_{\infty})^0=(0)$ if and only if $f_{\Lambda}Y(K_{\infty})$ is $\Lambda$-free.\medskip

Summarizing, we get the following criteria for the nullity of the obstruction kernels in the kummerian situation.
%%
%%proposition
%%
\begin{prop}\label{p:20-5} If $K$ contains $\mu_p$, then:
\begin{description}
\item[(i)] Under the Kuz'min-Gross conjecture, $\Psi(K_m)=(0)$ if and only if 
\begin{equation*}
t_{\Lambda}Y(K_{\infty})\cong W_{\infty}\bigoplus t_{\Lambda}\text{BP}(K_{\infty})\, ;
\end{equation*}
\item[(ii)] $X'(K_{\infty})^0=(0)$ if and only if $Y(K_{\infty})\cong t_{\Lambda}Y(K_{\infty})\bigoplus \Lambda^{r_2}$.
\end{description}
\end{prop}
We can now show a key ``propagation'' result for the vanishing of the obstrction kernels.
%%
%%proposition
%%
\begin{prop}\label{p:20-10}
For a number field $K$, let $K^{(2)}$ be the compositum of $K^{\text{cyc}}$ with an independent $\Zfield_p$-extension $F_{\infty}=\cup_{m}F_m$. Then the nullity of $\Psi(K^{\text{cyc}})$ propagates to all levels $F_m^{\text{cyc}}$.
\end{prop}
%%
%%proof
%%
\begin{proof} Since Galois descent works smoothly on $A'( \cdot )$ and $X'(\cdot )$ through $\Delta=\Gal(K(\mu_{p})/K)$ because $p$ is odd, we may as well suppose that $K$ contains $\mu_p$. Put $H_m=\Gal\big(F_m^{\text{cyc}}/K^{\text{cyc}}\big)$. For the first property, because $H_m$ acts trivially on $\Zfield_p(1)$ and by definition of $W\big(F_m^{\text{cyc}}\big)$, one has $W\big(K^{\text{cyc}}\big)\cong W\big(F_m^{\text{cyc}}\big)^{H_m}$. Moreover, Shapiro's lemma shows that $W\big(F_m^{\text{cyc}}\big)$ is $H_m$-cohomo\-logically trivial, and the $H_m$-cohomology of the exact sequence describing $t_{\Lambda}\text{BP}\big(K^{\text{cyc}}\big)$ in {\bfseries (P3)} implies immediately that 
\begin{equation*}
t_{\Lambda}\text{BP}\big(K^{\text{cyc}}\big)\cong t_{\Lambda}\text{BP}\big(F_m^{\text{cyc}}\big)^{H_m}.
\end{equation*}
 It follows that the $H_m$-fixed field of $\big(F_{m}^{\text{cyc},\text{BP}}\big)\cap N'\big(F_m^{\text{cyc}}\big)$ (with obvious notations, see property {\bfseries (P2)} above) is $K^{\text{cyc},\text{BP}}\cap N'\big(K^{\text{cyc}}\big)$, i.e. $\Psi\big(K^{\text{cyc}}\big)=\Psi\big(F_m^{\text{cyc}}\big)^{Hm}$. As $H_m$ is a $p$-group, the nullities of $\Psi\big(K^{\text{cyc}}\big)$ and of $\Psi\big(F_m^{\text{cyc}}\big)$ are equivalent.
\end{proof}
%%
%%exemples
%%
\begin{exs}\label{ex:20-5} Take $K=\Qfield(\mu_p)$, which admits a special $\Zfield_p^2$-extension of Greenberg type as in exemple~\ref{ex:5-1} of \S\ref{sec5}. Over $\Qfield(\mu_p)$ one can easily check that $M_{\infty}=K^{\text{BP}}_{\infty}$, $T_{\infty}={N'}_{\infty}$, $W\big(K^{\text{cyc}})=(0)$, hence $\Psi(K^{\text{cyc}})=(0)$. With the usual notations $(\cdot )^{\pm}$ for the $\pm$-eigenspaces of complex conjugation, recall that the classical \emph{GC} for $K^*$ is equivalent to the finiteness of $X'\big(K^{\text{cyc}}\big)^+$, and Vandiver's conjecture to the vanishing of $X'(K_{\infty})^+$. If we follow Vandiver, the nullity of $X'\big(K^{\text{cyc}}\big)^0$ is reduced to that of its minus part, but it is well known that the later property automatically holds for CM-fields (\cite{kiwa2}, \cite{ikuz1}). Summarizing, under Vandiver's conjecture (which has been checked up to $p<2^{27}$), proposition~\ref{p:20-10} applies.
\end{exs}
%%
%%section
%%
\section{\large On $(p,i)$-regular fields\label{sec25}}
To justify the introduction of this notion and the main theorem~\ref{t:25-5} of this section, it is not superfluous to recall the general relations between the conditions called {\bfseries (I)}, {\bfseries (II)}, {\bfseries (III)}, at the beginning of \S\ref{sec15}. Roughfly speaking \emph{GGC} is equivalent to {\bfseries (I)}+{\bfseries (II)}+{\bfseries (III)} and {\bfseries (II)}+{\bfseries (III)} imply {\bfseries (I)}, but {\bfseries (I)} and {\bfseries (III)} are independent \emph{a priori}. The $(p,i)$-regularity of the base field will impose strong dependence relations between  {\bfseries (I)}, {\bfseries (II)}, {\bfseries (III)} (see the end of the proof of theorem~\ref{t:25-5}) but, because precisely this strength, limited to their $e_{i-1}$-parts.\medskip

So our task is now to produce families of number fields to which the above (or slightly modified) condition of vanishing obstructions in theorem~\ref{t:10-5} could be applied using e.g. proposition~\ref{p:20-10}. We would like also to replace the $p$-class groups by more amenable $p$-adic cohomology groups. To this end let us introduce the particular class of the so called \emph{$(p,i)$-regular fields ($i\in\Zfield$)}, for which our main reference will be \cite{mkol2}. A special cases has been much studied $i=0$ (the so-called \emph{$p$-rational fields}, \cite{amov1}, \cite{mo_ng1}).The gist of $(p,i)$-regularity is that it kills a part of our obstruction kernel, see theorem~\ref{t:25-5} below. Let us first recall or introduce some notations for a general number field $K$: by a slight abuse of language, $S$ denotes the set of $p$-primes of any algebraic extension of $k$, $G_S(K)$ the Galois group over $K$ of the maximal algebraic extension of $K$ unramified outside $S$, $\Gcal_S(K)$ the maximal pro-$p$-quotient of $K_S(K)$, $Y(K)=\Gcal_S(K)^{\text{ab}}$, $Y(K^{\text{cyc}})=\varprojlim Y(K_n)$, $K'=K(\mu_p)$, $K'_{\text{cyc}}=$the cyclotomic $\Zfield_p$-extension of $K'$, $\Gamma=\Gal(K'_{\text{cyc}}/K')$, $\Lambda=\Zfield_p[[\Gamma]]$, $G=\Gal(K'_{\text{cyc}}/K$, $\Delta=\Gal(K'/K)$, cyclic of order $d$. Recall that for a Galois module $M$, the $j$-th Tate twist (with $j\in\Zfield$ is denoted $M(j)$.
%%
%%proposition-definition
%% 
\begin{prop-def}\label{p:25-5} For any $i\in\Zfield$, the number field $K$ is called $(p,i)$-regular if and only if  it satisfies one of the following equivalent conditions:
\begin{description}
\item[(1)] The Galois cohomology group $H^2\big(G_S(K),\Zfield/p\Zfield(i)\big)$ is trivial,
\item[(2)] The co-descent module $Y\big(K'_{\text{cyc}}\big)(-i)_{G}$ is   $\Zfield_p$-free,
\item[(3)] If $K$ contains $\mu_p$ or $i\equiv 1\mod d$, then $K$ admits only one $p$-place and $A'(K)$ is null. If $K$ does not contain $\mu_p$ and $i\not\equiv 1\mod d$, then $A'(K')(i-1)^{\Delta}$ is null, and for any $p$-place $v$ of $K$, the local degree $[K_v(\mu_p):K_v]$ does not divide $(i-1)$.
\end{description} 
If $i\neq 1$, there is also equivalence with any of the following:
\begin{description}
\item{(4)} The $p$-adic cohomology group $H^2\big(G_S(K),\Zfield_p(i)\big)$ is trivial,
\item{(5)} The module  $Y\big(K'_{\text{cyc}}\big)(-i)_{\Delta}$ is   $\Lambda$-free.
\end{description} 
\end{prop-def}
%%
%%proof
%%
\begin{proof}For the equivalence between properties {\bfseries (1)} and {\bfseries (4)}, see \cite{as_ng1}. For the equivalence between {\bfseries (4)} and{\bfseries (5)}, notice first that for $i\neq 1$, 
\begin{equation*}
t_{\Zfield_p}\big( Y\big(K'_{\text{cyc}}\big)(-i)_G\big)
\end{equation*}
 is the Pontryagin dual of $H^2\big(G_S(K),\Zfield_p(i)\big)$ (\cite{tngu1}, corollary~5.4), and that the property of $(p,i)$-regularity propagates along the cyclotomic tower (see {\bfseries (M2)} below). Going up the tower, we also know that $\varprojlim t_{\Zfield_p}\big(Y(K'_{\text{cyc}})(-i)_{\Delta}\big)_{\Gamma^{p^n}}\cong t_{\Lambda}\big(Y(K'_{\text{cyc}}\big)(-i)_{\Delta}$ (see \cite{ng_ni1}, proposition~3.1), hence condition {\bfseries (4)} implies that $Y\big(K'_{\text{cyc}}\big)(-i)_{\Delta}$ has no $\Lambda$-torsion. But more is true: from the structure theorem for a noetherian $\Lambda$-torsion-free module $M$, it is easy to show that $M$ is $\Lambda$-free if and only if $M_{\Gamma}$ is $\Zfield_p$-free, so {\bfseries (4)} implies {\bfseries (5)}. Conversely, condition  {\bfseries (5)} implies that $t_{\Zfield_p}\big(Y\big(K'_{\text{cyc}}\big)(-i)_{G}\big)\cong H^2\big(G_S(K),\Zfield_p(i)\big)=(0$.
\end{proof}

\paragraph{ Miscellaneaous properties of $(p,i)$-regular fields}
\begin{description}
\item[(M1)] By {\bfseries (1)} of proposition-definition~\ref{p:25-5},  $(p,i)$-regularity property is periodic $\mod d$. Of course, if $K$ contains $\mu_p$, then $d=1$ and all the properties of  $(p,i)$-regularity are equivalent to $p$-rationality (see \S\ref{sec30}). A typical example is $\Qfield_p(\mu_p)$, which is $p$-rational if and only if $p$ is a regular prime. One can take advantage of the periodicity $\pmod d$ to choose for each class $i$ a representative $i\geq 2$. Then the Quillen-Lichtenbaum conjecture (now a theorem of Voevodsky-Rost) asserts the existence of a canonical isomorphism $H^2\big(G_S(K),\Zfield_p(i)\big)\cong K_{2i-2}(O_K)\otimes\Zfield_p$ so that the condition {\bfseries (4)} of the proposition-definition~\ref{p:25-5} is equivalent to the \emph{nullity} of $K_{2i-2}(O_K)\otimes\Zfield_p$.
\item[(M2)] The property  of $(p,i)$-regularity goes up any normal $p$-extension $L/K$ which is unramified outside $(p)$. This gives an amply infinite number of $(p,i)$-regular fields.
\item[(M3)] For $i\neq 1$, it is easy to see that the finiteness of $H^2(G_S(K),\Zfield_p(i))$ is equivalent to the nullity of $H^2(G_S(K), \big(\Qfield_p/\Zfield_p\big)(i)$. The latter property is usually called \emph{$i^{\text{th}}$-twisted Leopoldt conjecture} (only for $i\neq 1$ because $H^2(G_S(K),\mu_{p^{\infty}})=\Br_S(K)\{p\}$ is not trivial in general), which is thus implied by $(p,i)$-regularity.
\end{description}

Keeping the notations of the previous sections, we now study \emph{GGC} for $(p,i)$-regular fields $K$, in the same spirit as the article \cite{ mkur1} (see also the example~\ref{ex:25-5} {\bfseries (3)} below for the \emph{``$e_i$-parts''} of the Vandiver conjecture. We shall use the expression ``modulo the Leopoldt conjecture'' in the same sense  as the previous expression ``modulo-the Kuz'min-Gross conjecture''.
%%
%%theorem
%%
\begin{thm}\label{t:25-5} Let  $k$ be an imaginary $(p,i)$-regular field admitting a special $\Zfield_p^2$-extension. Then modulo the Kuz'min-Gross and the Leopoldt conjecture, $e_{i-1}X'\big(\widetilde{k}(\mu_p)\big)$ is pseudo-null. In particular, if $k$ is $(p,1)$-regular, or if $k$ contains $\mu_p$ and is $p$-rational, then $k$ satisfies \emph{GGC}.
\end{thm}
%%
%%proof
%%
\begin{proof} Take a special $\Zfield_p^2$-extension $K^{(2)}/k$ as in theorem~\ref{t:10-5}, obtained by composing $k^{\text{cyc}}$ with an independent $\Zfield_p$-extension $F_{\infty}/k$ and introduce as before $k'=k(\mu_p)$, $F'_m=F_m(\mu_p)$, $\Delta=\Gal(k'/k)=\Gal(F'_m/F_m)$. As usual, for $j\in \Zfield$, consider the idempotent $e_j$ of $\Zfield_p[\Delta]$ attached the $j$-th power of the Teichm\"uller character. Recall that $M(j)^{\Delta}\cong M(j)_{\Delta}\cong e_jM$ and, by duality, $\big(M^{\vee}\big)(j)=\big(M(-j)\big)^{\vee}$, where the Galois action on the Pontryagin dual $M^{\vee}$ is defined by $(\sigma f)(m)=f(\sigma^{-1}(m))$. Let us show tha $e_{i-1}X'\big(K^{(2)}(\mu_p)\big)$ is pseudo null by considering two cases:
%%
%%paragraph
%%
\paragraph{(i)} If $i=1$, it follows from our hypotheses and proposition-definition~\ref{p:25-5} {\bfseries (3)} that $A'(K_{m,n})=0$ for all the intermediary fields $K_{m,n}$ (in the notations of the beginning of \S\ref{sec15}), so that $X'\big(K^{(2)}\big)=0$. This is a stronger result than needed, but note that it is not so easy to produce from scratch a field $K$ such that all $A'(K_{m,n})$ vanish (but see proposition~\ref{p:15-1}).
%%
%%paragraph
%%
\paragraph{(ii)} If $i\neq 1$, according to proposition-definition~\ref{p:25-5} (5), $k$ is $(p,i)$-regular if and only if $e_{-i}Y\big({k'}^{\text{cyc}}\big)$ is $\Lambda$-free, or equivalently $e_{-i}t_{\Lambda}Y\big({k'}^{\text{cyc}}\big)=(0)$ (with obvious adapted notations, everything happening now over $({k'}^{\text{cyc}}$).  After {\bfseries (P3)}, \S\ref{sec20}, the second condition is equivalent to $e_{i-1}X'\big({k'}^{\text{cyc}}\big)^0=0$ As for the first, applying $e_{-i}$ to the exact sequence decomposing  $t_{\Lambda}Y\big({k'}^{\text{cyc}}\big)$ at the end of {\bfseries (P2)}, we see that  $e_{-i}t_{\Lambda}Y\big({k'}^{\text{cyc}}\big)=(0)$ implies that 
\begin{equation*}
e_{-i}t_{\Lambda}\text{BP}\big({k'}^{\text{cyc}}\big)=(0)=e_{-i}\Hom\big(\Psi\big({k'}^{\text{cyc}}\big),\mu_{p^{\infty}}\big)\, .
\end{equation*} 
 But we know also from {\bfseries (P2)} that the nullity of $e_{i-1}\Psi\big({k'}^{\text{cyc}}\big)$ implies that
\begin{equation*}
e_{-i}\Gal\big(M^{\text{cyc}}/{N'}^{\text{cyc}}\big)\cong e_{-i}\Gal\big({K'}^{\text{BP}}_{\text{cyc}}/T^{\text{cyc}}\big)=e_{-i}t_{\Lambda}\text{BP}\big({k'}^{\text{cyc}}\big)\, ,
\end{equation*}
and thus the $(p,i)$-regularity hypothesis is equivalent to 
\begin{equation*}
e_{i-1}\Psi\big({k'}^{\text{cyc}}\big)=e_{i-1}X'\big({k'}^{\text{cyc}}\big)^0=0\, .
\end{equation*}
At this point, it remains to deal with the $e_{i-1}$-part of $A'(k')$. Kummer duality in Iwasawa theory shows that 
\begin{equation*}
\Gal\big(M^{\text{cyc}}/{N'}^{\text{cyc}}\big)\cong \Hom\big(A'\big({k'}^{\text{cyc}}\big),\mu_{p^{\infty}}\big),
\end{equation*}
 hence $e_{i-1}A'\big({k'}^{\text{cyc}}\big)=(0)$ because $(p,i)$-regularity implies 
 \begin{equation*}
 e_{-i}\Gal\big(M^{\text{cyc}}/ {N'}^{\text{cyc}}\big)\\=(0),
 \end{equation*}
 as we have just seen. 

But a result of Ozaki, \cite{moza1} (for an update see \cite{tngu4}, proposition ~2.2), states that the natural image of $X'\big({k'}^{\text{cyc}}\big)^{0}$ in $A'(k')$ coincide with ${\rm Cap}\big({k'}^{\text{cyc}}/k'\big)$, hence $e_{i-1}{\rm Cap}\big({k'}^{\text{cyc}}/k'\big)=(0)$, i.e.  $e_{i-1}A'(k')$ injects into $A'\big({k'}^{\text{cyc}}\big)$. Summarizing, $(p,i)$-regularity simultaneously annihilates the $e_{i-1}$-part of $A'(k')$, $\Psi\big({k'}^{\text{cyc}}\big)$ and  $X'\big({k'}^{\text{cyc}}\big)^{0}$, and the conditions in theorem~\ref{t:10-5} apply to the $e_{i-1}$-part of  the relevant objects over ${k'}^{\text{cyc}}$. Under our hypotheses, proposition~\ref{p:20-10} shows that  $(p,i)$-regularity of $k$ propagates to any $F_m$ and so $X'\big(K^{(2)}(\mu_p)\big)$ is pseudo-null. The passage from $X'\big(K^{(2)}(\mu_p)\big)$ to$X'\big(\widetilde{k}(\mu_p)\big)$ then proceeds as in theorem~\ref{t:5-5}, by climbing up the  $e_{i-1}$-part of a tower of green multiple $\Zfield_p$-extensions.
\end{proof}
%%
%%Exemples
%%
\begin{exs}\label{ex:25-5} \hfill\\
%%
%%paragraph
%%
\vspace{-10mm}\paragraph{(1)} The $p$-th cyclotomic field $k=\Qfield(\mu_p)$ admits a special $\Zfield_p^2$-extension of Greenberg type as in the example~1 of \S\ref{sec5} and theorem~\ref{t:15-5} applies. By the miscellaneous property {\bfseries (M1)}, $\Qfield(\mu_p)$ is $p$-rational if and only if $p$ is a regular prime, in which case $X'(\widetilde{k})$ is pseudo-null (it is actually null). But unfortunately we do not know whether there exists an infinite number of regular primes.
%%
%%paragraph
%%
\paragraph{(2)} Keep $k=\Qfield(\mu_p)$. Since $H^2(G_S (k^+),\Zfield/p\Zfield)= H^2(G_S( k),\Zfield/p\Zfield)^+$, the $p$-rationality of $\Qfield(\mu_p^+)$ is equivalent to the Vandiver conjecture . In fact, denoting by $A$ the $p$-class group of $\Qfield(\mu_p)$ and $e_i$ the usual idenpotent of $\Gal(\Qfield(\mu_p)/\Qfield)$, it is known that for $0\leq i < p-1$, $e_i(A/pA)\cong H^2\big(G_S\big(\Qfield(\mu_p)\big), \Zfield/p\Zfield(1-i)\big)$ (\cite{ikuz1}, lemma~1.2). The Vandiver conjecture is thus equivalent to the vanishing of these cohomology groups for all even $0\leq  i < p-1$, and then theorem~\ref{t:25-5} implies the pseudo-nullity of the plus part of $X'(\widetilde{k})$. The occurence here of the $p$-class group in place of the higher \'etale $K$-groups $(i\geq 2)$ comes from the fact that $\Qfield\big(\mu_{p^{\infty}}\big)$ admits only one $p$-place so that $t_{\Lambda}Y\big(\Qfield\big(\mu_{p^{\infty}}\big)\big)=t_{\Lambda}{\rm BP}\big(\Qfield(\mu_{p^{\infty}})\big)$.
%%
%%paragraph
%%
\paragraph{(3)} The adequation of Vandiver's conjecture with the settings of theorem~\ref{t:25-5} is not fortuitous. Although the base field here is not totally real, it is natural to wonder where the classical \emph{GC} may come into play. Let us go back to the general situation and notations of the properties {\bfseries (P1)} to {\bfseries (P3)} in \S\ref{sec15}. Our main reference will be the appendix of \cite{ng_ni1}.

 The natural map $t_{\Lambda}Y\big({k'}^{\text{cyc}}\big)\rightarrow X'\big({k'}^{\text{cyc}}\big)$ induces a surjective map
\begin{equation*}
t_{\Lambda}{\rm BP}\big({k'}^{\text{cyc}}\big)\rightarrow Z\big({k'}^{\text{cyc}}\big):=\Gal\big({L'}^{\text{cyc}}/({L'}^{\text{cyc}}\cap {T}^{\text{cyc}} \big)\, .
\end{equation*}
In the CM case, 
\begin{equation*}
Z\big({k'}^{\text{cyc}}\big)=Z\big({k'}^{\text{cyc}}\big)^+\cong X'\big({k'}^{\text{cyc}}\big)^+\bigoplus \alpha\big(X'\big({k'}^{\text{cyc}}\big)\big)(-1)\, ,
\end{equation*}
where $\alpha(\cdot )$ denote the adjoint; it follows that the classical \emph{GC}, which predicts the finiteness of  $X'\big({k'}^{\text{cyc}}\big)^+$, is equivalent to the \emph{finiteness} of $Z\big({k'}^{\text{cyc}}\big)$, because the adjoint has no non trivial finite submodule. This allows to propose an extended Greenberg conjecture  \emph{EGC}(=conjecture~6.0.1 in \cite{ng_ni1}), namely the \emph{general} finiteness of $Z\big({k'}^{\text{cyc}}\big)$. When the base field is $k=k'=\Qfield(\mu_p)$, Vandiver's conjecture is equivalent to $Z\big({k'}^{\text{cyc}}\big)=X'\big({k'}^{\text{cyc}}\big)^+=(0)$ (which implies the condition {\bfseries (i)} of proposition~\ref{p:15-1}). This also suggests to call ``Vandiver's property'' for a CM field $k$ (not a conjecture !) the nullity of $Z\big({k'}^{\text{cyc}}\big)$.
%%
%%paragraph
%%
\paragraph{(4)}  One could also examine directly what happens when $k'$ is a $p$-rational field which is the maximally totally real subfield  of $K=k'(\mu_p)$, with $\Delta=<c>$ generated by the complex conjugation. Assume also  that $K$ admits a special $\Zfield_p^2$-extension $K^{(2)}/K$ which is of  Greenberg type (see example~\ref{ex:5-1}, \S\ref{sec5}) where  $K^{(2)}$ is the compositum of the cyclotomic$\Zfield_p$-extension $K^{\text{cyc}}$ and an independant $\Zfield_p$-extension $F_{\infty}=\cup F_m$, with $\Gal(K/\Qfield)$ acting on $\Gal(F_{\infty}/K)$ by multiplication by $-1$. Denoting  by $(.)^{\pm}$ the eigenspaces of $c$, we have $H^2(G_S(K),\Zfield/p\Zfield)^+=H^2(G_S(k'),\Zfield/p\Zfield)=(0)$ by the $p$-rationality of $k'$. With the notations of proposition~\ref{p:25-5} this means the nullity of $\big(t_{\Lambda}Y\big(K^{\text{cyc}}\big)\big)^+$, where $\Lambda $ is the Iwasawa algebra of $\Gamma=\Gal\big(K^{\text{cyc}}/K\big)\cong\Gal\big(F_m^{\text{cyc}}/F_m\big)$. By construction, the extension $F_m/k'$ is no longer abelian but dihedral, more precisely, the group $G_m=\Gal(F_m/k')$ is generated by $H_m=\Gal(F_m/K)=<\sigma>$ and $\nabla=<\tau>$, where $\tau$ is a (non canonical) lift of $c$ such that $\tau \sigma=\sigma^{-1}\tau$ (it is again recommended to draw a Galois diagram). As a (non normal) subgroup of $G_m$, $\nabla$ acts on any $G_m$-module  $M$, and we write $M^{(\pm 1)}$for the eigenspaces of $M$ corresponding to the eigenvalues  $\pm 1$ (in additive notation).  Here  we prefer  to avoid the notation $M^{\pm}$ because the $M^{(\pm 1)}$ are not necesseraly stabilized by the whole $G_m$ since $F_m$ is not CM; but note that above $F_m$ we are allowed to use the (non canonical) $(.)^{(\pm 1)}$-parts of $\Lambda$-modules because any lift of $\tau$ to the level of $F_m^{\text{cyc}}$commutes with $\Gamma$. Descending from $F_m$ to $K$ in order to exploit the CM hypothesis  on $K$, and taking $H_m$-invariants to trivialize the action of $\sigma$, we get $(M^{H_m})^{\pm}\cong (M^{H_m})^{(\pm 1)}$, and since $H_m$ is a $p$-group, the nullity of both $(M^{H_m})^{\pm}$ will be equivalent  to that of $M$. Here the vanishing of $\big(X'\big(F_m^{\text{cyc}}\big)^0\big)^{(+1)}$ will follow from the $p$-rationality of $k'$, and that of $\big(X'\big(F_m^{\text{cyc}}\big)^0\big)^{(-1)}$ from the CM hypothesis on $K$.\medskip

More generally, one could stick to the line of \cite{ikuz1}, fix a base field $k$ which is $p$-rational (and plays the role of $\Qfield$) and try to give conditions for a subextension of $k(\mu_p)/k$ to verify \emph{GGC}. For examples for families of $p$-rational fields, see references below.
%%
%%paragraph
%%
\paragraph{(5)} To produce examples outside the setting of theorem~\ref{t:15-10}, one could replace $\Qfield(\mu_p)$ by $k(\mu_p)$, where $k$ is a number field in which $p$ splits, and the $p$-class group by the higher \'etale $K$-groups (see \cite{tngu3}, theorem~1.1.2). For $i\geq 2$, we have  seen above that $(p,i)$-regularity corresponds to the vanishing of the $p$-primary part of the \emph{global} object $K_{2i-2}(O_k)$.
\end{exs}

%%
%%section
%%
\section{\large Appendix on $p$-rational fields\label{sec30}}
We would like to briefly compare our approach and Fujii's for the construction of a green $\Zfield_p^2$-extension. Recall that a $(p,0)$-regular field is called $p$-rational in \cite{mo_ng1}, the terminology being meant to suggest that such a field ``behaves like  $\Qfield$'' at the prime $p$.
%%
%%proposition
%%
\begin{prop}\label{p:30-5} A number field $K$ is $p$-rational if and only if it satisfies one of the following equivalent conditions
\begin{description}
\item[(i)] The pro-$p$-group $\Gcal_S(K)$ is free, i.e. $H^2(\Gcal_S(K),\Zfield/p\Zfield)(=H^2(G_S(K),\\ \Zfield/p\Zfield))$ is null;
\item[(ii)] $K$ verifies Leopoldt's conjecture at $p$ and $\Gcal_S(K)^{\text{ab}}$ is $\Zfield_p$-free (necessarily of rank $1+r_2$); 
\item[(iii)] The cokernel $W_S=W_s(K)$ of the diagonal map  $\mu_{p^{\infty}}\rightarrow \bigoplus _{v\in S}\mu_{p^{\infty}}(K_{v})$
is null and $A'\big(K(\mu_p)\big)(-1)^{\Delta}=(0)$ with $\Delta=\Gal\big(K(\mu_p)/K\big)$;
\item[(iv)] The $p$-Hilbert class-field ${\rm Hilb}_K$ of $K$ is contained in the composite $\tilde{K}$ of all the $\Zfield_p$-extensions of $K$, $W_S=(0)$, and the diagonal map $U_K/p\rightarrow \bigoplus_{v\in S}U_{v}/p$, where $U_K$ (resp. $U_{v}$) denotes the group of units of $K$ (resp. $K_{\nu}$, is injective.
\end{description}
\end{prop}
%%
%%proof
%%
\begin{proof} The equivalence between properties {\bfseries (i)} to {\bfseries (iii)} is known and can be found in \cite{mo_ng1}. Let us show only the equivalence between {\bfseries (ii)} and {\bfseries (iv)}. Consider the following Galois situation: $A(K)\cong \Gal({\rm Hilb}_K/K)=\Gcal_S(K)^{\text{ab}}$, $t(\cdot )$ denotes $\Zfield_p$-torsion and $\varphi: \, tY(K)\rightarrow A(K)$ is induced by the natural projection  $Y(K)\rightarrow A(K)$. The nullity of $tY(K)$ is equivalent to that of both $\Ima \varphi$ and $\Ker \varphi$. Since $\Ima\varphi=\Gal\big({\rm Hilb}_K/\widetilde{L}\cap ({\rm Hilb}_K\big)$, its nullity means that ${\rm Hilb}_K\subset \widetilde{K}$. To study the nullity of  $\Ker\varphi$, let us write  down the exact sequence of class field theory relative inertia
\begin{equation*}
\begin{tikzcd}[sep=small]
0\arrow[r] & \overline{U}_K\arrow[r] & \bigoplus_{v\in S}\overline{U}_v\arrow[r] & Y(K)\arrow[rr] \arrow[dr]& &A(K)  \arrow[r]  &  0\\
& & & & I(K)\arrow[ur]&
\end{tikzcd}
\end{equation*}
where the injectivity on the left is due to Leopoldt's conjecture, and $I(K)$ denotes the inertia subgroup at all $p$-places. For $n$ large enough, the snake lemma for multiplication by $p^n$ gives two exact sequences 
\begin{equation*}
0\rightarrow W_S\rightarrow tI(K)\rightarrow U_K/p^n\xrightarrow{\delta_n}\bigoplus_{v\in S} U_v/p^n
\end{equation*}
and
\begin{equation*}
0\rightarrow tI(K)\rightarrow tY(K)\xrightarrow{\varphi}A(K)
\end{equation*}
Hence  $\Ker\varphi=(0)$ if and only if  $tI(K)=(0)$, if and only if $W_S=(0)$ and $\delta_n$ is injective. Assuming $W_S=(0)$, it is clear that the injectivity of $\delta_n$ is equivalent to that of 
\begin{equation*}
{\rm fr\,}U_K/p^n\longrightarrow\bigoplus _{v\in S}{\rm fr\,}U_v/p^n\, , 
\end{equation*}
where ${\rm fr\,}(.)$ denotes the quotient modulo $\Zfield_p$-torsion. Since ${\rm fr\,}U_K$ and ${\rm fr\,}U_v$ are lattices, this is in turn equivalent to the injectivity of ${\rm fr\,}U_K/p\rightarrow \bigoplus_{v\in S}{\rm fr\,}/pU_v$, hence, still assuming that $W_S=(0)$, equivalent to the injectivity of $\delta_1$. Actually the previous argument shows  that the nullity of $W_S$ and the injectivity of $\delta_1$ imply the injectivity of all $\delta_n$, which is in turn a particular case of one of the numerous formulations of Leopoldt's conjecture for $K$: there exists a constant $c=c(K)$ such that for all $n\gg 0$, for all $u \in U_K$, one has $u\in {U_K}^{p^n}$ if $u\in {U_v}^{p^{n+c}}$ for all $v\in S$ (see e.g. \cite{as_ng1} proposition~1). The equivalence bettween {\bfseries (ii)} and {\bfseries (iv)} is thus proved. 
\end{proof}

Recall \emph{Itoh's condition's} for a CM field $k$:
\begin{description}
\item[(I1)] the odd prime number $p$ splits completely in $k$,
\item[(I2)] Leopoldt's conjecture holds for $p$ and $k^+$,
\item[(I3)] $p$ does not divide the class number of $k$,
\item[(I4)] the module $X(k^+_{\text{cyc}})$ is trivial (this is stronger that \emph{GC}).
\end{description}
%%
%%proposition
%%
\begin{prop}\label{p:30-10} A CM-field $k$ which satisfies  Itoh's conditions is $p$-rational (and in particular verifies Leopoldt's conjecture).
\end{prop}
%%
%%proof
%%
\begin{proof} The usual notations $\Gamma$, $\Lambda,$ etc. (such as in the proof of proposition~\ref{p:30-5}) will pertain to the cyclotomic  extension $k^+_{\text{cyc}}$ of $k^+$. In particular, $Y(k^+_{\text{cyc}})=\varprojlim Y(k^+_n)$ and $X(k^+_{\text{cyc}})=\varprojlim A(k^+_n)$. We proceed in two steps:
%%
%%paragraph
%%
\paragraph{(a)}\emph{$k^+$ is $p$-rational}: Leopoldt's conjecture for $k^+$  implies obviously that $Y(k^+_{\text{cyc}})_{\Gamma}\cong tY(k^+)$. Let $M$ be the extension of $k^+_{\text{cyc}}$ s.t. $\Gal(M/k^+_{\text{cyc}})=Y(k^+_{\text{cyc}})_{\Gamma}$. Under condition~{\bfseries (I1)}, a result of Ozaki and Taya (\cite{htay1}, lemma~2.3; for another proof of an equivalent property , see \cite{ba_ng1}, lemma~3.1) asserts that $M/k^+_{\text{cyc}}$ is unramified so that  $M=k^+_{\text{cyc}}$ under {\bfseries (I4)}, hence $tY(k^+)$ is trivial , i.e. $k^+$ is $p$-rational.
%%
%%paragraph
%%
\paragraph{(b)} \emph{$k$ is $p$-rational}: We use the last criterion of proposition~\ref{p:30-5}. Condition {\bfseries (I3)} is stronger than the condition ${\rm Hilb}_k\subset \tilde{k}$. It remains only to study the injectivity of the diagonal map  $U_k/p\rightarrow \bigoplus_{v\in S}U_v/p$. The ``plus'' part is taken care of by the $p$-rationality of $k^+$. Concerning the ``minus'' part, just notice that ${U_k}^-$ consists of roots of unity because $k$ is CM, and the local roots of unity have order dividing $(p-1)$ because of the $p$-splitting condition {\bfseries (I1)}.
\end{proof}

As was recalled at the beginning of section~\ref{sec10}, taking a base field $k$ which is CM and verifies Itoh's conditions, Fujii constructed in \cite{sfuj1} a sequence of (uniquely defined) multiple $\Zfield_p$-extensions \begin{equation*}
k\subset K^{(1)}\subset\cdots \subset K^{(d)}\subset\cdots \tilde{k}\, , 
\end{equation*}
using the $p$-splitting condition {\bfseries (I1)} in a crucial manner. Note that Fujii's $K^{(2)}$ does not contain $k^{\text{cyc}}$. In \cite{sfuj1}, step 1 of section~3, $X(K^{(1)})$ is shown to be trivial. In step 2, the pseudo-nullity of $X(K^{(2)})$ is proved  in the following way: let $L_2$ be the maximal abelian pro-$p$-extension of $K^{(2)})$ which is abelian over $K^{(1)}$, so that $\Gal(L_2/K^{(2)})\cong X(K^{(2)})_{\Gamma}$, where $\Gamma=\Gal(K^{(2)}/K^{(1)})$. Let $K^{(2)}=\cup K_m$. By hard class field theoritic calculations, Fujii shows that for $m$ large, the extension  $L_2/K_m$ is abelian of rank $2$, which means that $X(K^{(2)})_{\Gamma}$ is finite , and the sufficient criterion in lemma ~\ref{l:5-5} can be applied. But note that Fujii's result and ours here are of a different nature  because $k^{\text{cyc}}$ does not play the same role in the two settings.\medskip

\emph{Aknowledgements}. My hearty thanks to my colleague A. Movahhedi for many useful exchanges.
\renewcommand\refname{\normalsize References \textnormal{(intentionnally restricted to the commutative setting)}}

\bigskip

Universit\'e de Franche Comt\'e\\
CNRS UMR 6623\\
25030 Besançon Cedex France
\end{document}